\documentclass[12pt, reqno]{amsart}
\usepackage{amsmath, amsthm, amscd, amsfonts, amssymb, graphicx, xcolor}
\usepackage[bookmarksnumbered, colorlinks, plainpages,
citecolor=teal,
linktoc=all]{hyperref}

\newtheorem{theorem}{Theorem}[section]
\newtheorem{lemma}[theorem]{Lemma}
\newtheorem{proposition}[theorem]{Proposition}
\newtheorem{corollary}[theorem]{Corollary}
\newtheorem{definition}[theorem]{Definition}
\newtheorem{example}[theorem]{Example}

\theoremstyle{remark}
\newtheorem{remark}[theorem]{Remark}
\numberwithin{equation}{section}

\usepackage[all,pdf]{xy}
\usepackage{amssymb}
\usepackage{fancyhdr}
\usepackage{relsize}
\usepackage{enumerate}

\usepackage{CJKutf8}

\usepackage{subfigure}

\usepackage{hhline}

\usepackage{tikz}
\usetikzlibrary{arrows.meta}
\usetikzlibrary{positioning}
\usetikzlibrary{shapes}
\usetikzlibrary{calc}

\newcommand{\RR}{\mathbb{R}}

\newcommand{\ZZ}{\mathbb{Z}}
\newcommand{\CC}{\mathbb{C}}

\newcommand{\ii}{\mathbf{i}}
\newcommand{\ee}{\mathrm{e}}

\renewcommand{\AA}{\mathbb{A}}
\newcommand{\PP}{\mathbb{P}}
\newcommand{\FF}{\mathbb{F}}

\newcommand{\abs}[1]{\left\lvert #1 \right\rvert}
\newcommand{\defset}[2]{ \left\{\ #1 \ \left\rvert\ #2 \ \right\}\right. }

\newcommand{\rmd}{\mathrm{d}}
\newcommand{\pt}{\partial}

\newcommand{\PH}{Poincar\'{e}-Hopf}
\newcommand{\teich}{Teichm\"{u}ller}
\newcommand{\Kahler}{K\"{a}hler}

\newcommand{\qdf}{quadratic differential}

\newcommand{\nbhd}{neighborhood}
\newcommand{\homeo}{homeomorphism}

\newcommand{\ortprehomeo}{orientation-preserving homeomorphism}

\newcommand{\SYM}{\textcolor{black}}
\newcommand{\DEF}{\textbf}

\begin{document}
\setcounter{page}{1}


\centerline{}

\title[Moduli space of HCMU surfaces]{The  moduli space of HCMU surfaces}

\author[S.C. Lu, B. Xu]{Sicheng Lu$^1$ and Bin Xu$^2$} 

\address{$^{1}$ The Institute of Geometry and Physics, University of Science and Technology of China, Hefei, Anhui, People's Republic of China.
}
\email{\textcolor[rgb]{0.00,0.00,0.84}{sichenglu@ustc.edu.cn}}

\address{$^{2}$ School of Mathematical Sciences and CAS Wu Wen-Tsun Key Laboratory of Mathematics, University of Science and Technology of China, Hefei, Anhui, People's Republic of China.
}
\email{\textcolor[rgb]{0.00,0.00,0.84}{bxu@ustc.edu.cn}}




\date{ \today
}

\begin{abstract}
HCMU surfaces are compact Riemann surfaces equipped with an extremal K\"{a}hler metric and a finite number of singularities. Research on these surfaces was initiated by E. Calabi and X.-X. Chen over thirty years ago.
We provide a detailed description of the geometric structure of HCMU surfaces, building on the classical football decomposition introduced by Chen-Chen-Wu \cite{CCW05}. 
From this perspective, most HCMU surfaces can be uniquely described by a set of data that includes both discrete topological information and continuous geometric parameters.
This data representation is effective for studying the moduli space of HCMU surfaces with specified genus and conical angles, suggesting a topological approach to this topic. As a first application, we present a unified proof of the angle constraints on HCMU surfaces. Using the same approach, we establish an existence theorem for HCMU surfaces of any genus with a single conical point, which is also a saddle point.
Finally, we determine the dimension of the moduli space, defined as the number of independent continuous parameters. This is achieved by examining several geometric deformations of HCMU surfaces and the various relationships between the quantities in the data set representation.
\end{abstract} 

\maketitle

\tableofcontents

\section{Introduction}\label{sec:HCMU}
It is a natural problem to find a ``best'' metric in a conformal class of a Riemann surface, under certain measurement of ``cost''. This is another attempt to generalize the classical uniformization theorem to the case including singularities.
In this manuscript we focus on a class of such \Kahler\ metrics, called HCMU metrics, which were originally initiated by E. Calabi \cite{Calabi82extrm, Calabi85extrm} and X.-X. Chen \cite{Cxx98, Cxx99, Cxx00}, and developed
in \cite{WangZhu00, LinZhu02} etc. We shall study their geometric and topological properties, as well as moduli spaces that they form.

\subsection{Extremal K\"{a}hler metric and HCMU surface}\label{ssec:def_hcmu}

The original problem is proposed for compact complex manifold $\mathcal{M}$ of arbitrary dimension, by Calabi in his celebrated papers \cite{Calabi82extrm, Calabi85extrm}.
In a fixed \Kahler\ class, an \DEF{extremal \Kahler\ metric} is a critical point of the \DEF{Calabi energy} 
\begin{equation*}\label{eq:clb_eng}
     E(\rho) := \int_{\mathcal{M}} R^2 \ \rmd \rho
\end{equation*}
where $R$ is the scalar curvature of the metric $\rho$ in the \Kahler\ class. The Euler-Lagrange equations of $E(\rho)$ can be expressed as 
\begin{equation*}\label{eq:EL_eq_ext}
    R_{,\alpha\beta} = 0,\quad 1 \leq \alpha, \beta \leq \dim_\CC \mathcal{M} \ ,
\end{equation*}
where $R_{,\alpha\beta}$'s are the second-order  covariant derivatives of $R$. 

We will restrict to the case of $\dim_\CC\mathcal{M}=1$. 
Calabi \cite{Calabi82extrm} proved that an extremal \Kahler\ metric on a compact Riemann surface $M$ must has constant scalar curvature (\DEF{CSC}). This coincides with the classical uniformization theorem. 
Hence, to find extremal \Kahler\ metrics with non-constant scalar curvature (\DEF{non-CSC}), we must turn to singular metrics. X.-X. Chen in \cite{Cxx99} first gives an example of a non-CSC extremal \Kahler\ metric with singularities. 

\medskip
Let $\SYM{P}:=\{ p_1, \cdots p_n\}$ be a finite set on the compact Riemann surface $M$, and $\SYM{\vec{\alpha}} := (\alpha_1,\cdots,\alpha_n) \in (\RR_{>0} \setminus \{1\})^n$ be an \DEF{angle vector}. 
A conformal metric $\rho$ on $M$ is said to have a \DEF{conical singularity} at $p_i$ of cone angle $2\pi\alpha_i$ if there exists a local complex coordinate $z$ with $z(p)=0$ such that 
$ \rho = \ee^{2\varphi(z)} \abs{\rmd z}^2 $ on the punctured \nbhd\ of $0$, and that
\begin{equation}\label{eq:defn_cone}
    \varphi(z) - (\alpha_i-1) \ln \abs{z}
\end{equation}
is continuous at $0$ and smooth outside $0$. 
With this coordinate we have
$$ \Delta_\rho = \ee^{-2\varphi}\left( \frac{\pt^2}{\pt x^2} + \frac{\pt^2}{\pt y^2} \right) ,\quad 
\SYM{K} = - \Delta \varphi \cdot \ee ^{2\varphi} 
$$
where $K$ is the Gaussian curvature (equivalent to scalar curvature on Riemann surface). 
The Euler-Lagrange equation of the Calabi energy in a suitable Sobolev space of conformal metrics with finitely many singularities
gives rise to 
\begin{equation}\label{eq:EL_eq_RS}
    \Delta_{\rho}K + K^2 = C 
\end{equation}
with some constant $C$ (\cite{Calabi85extrm, Cxx99}). Or equivalently, 
\begin{equation}
    \frac{\pt}{\pt \overline{z}} K_{,zz} = 0 \ ,
\end{equation}
where $K_{,zz}$ is the second-order  covariant derivatives of $K$ (then $K_{,zz} \rmd z^2$ is a holomorphic \qdf\ on $M\setminus P$). See \cite{Cxx00} for details.


\begin{definition}[HCMU metric]\label{defn:HCMU}
We call $\rho$ an \DEF{HCMU metric} \index{HCMU!- metric} (the Hessian of the curvature of the metric is umbilical) if it is non-CSC and
\begin{equation}\label{eq:HCMU_0}  
    K_{,zz} =0 \ . 
\end{equation}    
Or equivalently (\cite[\S 2.1]{CCW05}), when $\rho$ is expressed as $\rho=\ee^{2\varphi(z)} \abs{\rmd z}^2$,
\begin{equation}\label{eq:HCMU}
    \frac{\pt^2 K}{\pt z^2} - 2 \frac{\pt K}{\pt z} \frac{\pt \varphi}{\pt z} =0.
\end{equation}
A Riemann surface $M$ endowed with an HCMU metric $\rho$ will be called an \DEF{HCMU surface}, denoted by $(M,\rho)$.  
\end{definition}

We always assume that an HCMU metric has finite area and finite Calabi energy, that is,
\begin{equation}
    A(\rho) = \int_{M \setminus P} \rmd \rho < +\infty, \quad
    E(\rho) = \int_{M \setminus P} K^2 \ \rmd \rho < +\infty \ .
\end{equation}

\begin{remark}
    Our study pays more attention on the geometric and topological aspect of HCMU surfaces and their moduli space. So we are not fixing the complex structure of underlying Riemann surfaces. 
\end{remark}

Another type of singularity is also concerned. By \eqref{eq:defn_cone},  
$\lim_{z\to 0} \frac{ \varphi(z)}{ \ln \abs{z} } = \alpha_i -1 $. 
Now we say that $\rho$ has a \DEF{cusp singularity} at $p_i$  if
\begin{equation}
    \lim_{z\to0} \frac{ \varphi(z)+\ln \abs{z} }{ \ln \abs{z} } = 0 \ . 
\end{equation} 
Hence a cusp singularity can be regarded as a conical singularity with angle zero. 

\bigskip
\subsection{Character 1-form}
Every HCMU metric induces a meromorphic 1-form on the underlying Riemann surface, which encodes many geometric information. 
Accordingly, ideas in the study of meromorphic differentials may be transplanted to HCMU surfaces. That is the starting point and motivation of this work, although differential itself is not seriously used later. 

\medskip
Given an HCMU surface $(M,\rho)$, let $\SYM{\mathfrak{C}} := \{\ p \in M\ \lvert\ \rmd K(p)=0 \ \}$ be the set of critical points of the curvature function. 

\begin{proposition}[\cite{Cxx00, CCW05}]\label{prop:vector_fields}
    The complex gradient vector field 
    \[ \SYM{{\rm grad}^{(1,0)} K} := \ii \mathrm{e}^{-2\varphi} \frac{\partial K}{\partial \overline{z}} \frac{\partial}{\partial z} \] 
    is holomorphic and has no zeros on $M \setminus \mathfrak{C}$. The integral curves of its imaginary part \SYM{$\vec{H}$} are geodesics.
    The real part \SYM{$\vec{V}$} of ${\rm grad}^{(1,0)} K$  is a Killing vector field perpendicular to $\vec{H}$, and its integral curve is the level set of $K$ 
\end{proposition}

\begin{definition}[\cite{CqWyy11, CWX15}]
    We call the dual 1-form $\omega$ of ${\rm grad}^{(1,0)} K$ the \DEF{character 1-form} of $\rho$.
    More precisely, $\omega$ is a meromorphic 1-form defined on $M$ such that $\omega({\rm grad}^{(1,0)} K) = \ii/4$, where $\ii/4$ is a technical constant.
\end{definition}

The character 1-form for HCMU metrics with conical singularities is originally studied in \cite{CqWyy11}, later in \cite{CWX15} for metrics with cusp singularities. 
We package results of the two papers in the following proposition.
\begin{proposition}[\cite{CqWyy11, CWX15}]
    Let $K$ be the curvature function of an HCMU surface $(M, \rho)$, whose maximum and minimum are $K_0, K_1$ respectively with
    \begin{equation}
        K_0>0, \quad K_0 > K_1 \geq -\frac{K_0}{2} \ .
    \end{equation}
    Then 
    the character 1-form $\omega$ on $M$ and has the following properties:
    \begin{enumerate}
        \item $\omega$ is meromorphic and only has simple poles,
        \item the residue of $\omega$ at each pole is a real number,
        \item $\omega+\overline{\omega}$ is exact outside the poles of $\omega$.
    \end{enumerate}
    Furthermore, the metric $\rho$, the curvature $K$ and the character 1-form $\omega$ satisfy:
    \begin{equation}
    \renewcommand\arraystretch{1.3}
        \left\{ \begin{tabular}{ll}
             $\rmd K$ & $= -\frac13 (K-K_0)(K-K_1)(K+K_0+K_1) \cdot(\omega+\overline{\omega})$ , \\
             $g$ & $= -\frac43 (K-K_0)(K-K_1)(K+K_0+K_1) \cdot\omega\overline{\omega}$ , \\
             $K(p)$ & $\in (K_1,K_0)$ , for any $p \in M$ other than the poles of $\omega$. \\
        \end{tabular} \right.
    \end{equation}
\end{proposition}

The curvature function $K$ together with two vector fields $\vec{H}, \vec{V}$ encodes many geometric properties of HCMU surfaces. We summarize some of them in the following. They lead to the football decomposition (\cite{CCW05}, also see Section \ref{ssec:football_decomp}), which is the base point of this work.
\begin{proposition}[\cite{Cxx98, Cxx99, Cxx00, LinZhu02, CCW05, CWX15}]\label{prop:curvature}
    Let $(M,\rho)$ be an HCMU surface with finite conical or cusp singularities, and $K$ be its curvature function. 
    \begin{enumerate}
        \item $K$ extends continuously over the singularities.
        \item The values of local minimum (or maximum) of $K$ are all the same. 
        \item The limit of $K$ at any cusp singularity is negative and achieves minimum.
        \item The singularities of $\vec{V}$ or $\vec{H}$ is the finite union of all the singularities of metric $\rho$ and smooth critical points of $K$. Any saddle point of $K$ must be a conical singularity with $2k\pi\ (k\in\ZZ_{>1})$ angle of the HCMU metric.
    \end{enumerate}    
\end{proposition}
 
\begin{definition}
     The points with local maximum (minimum) curvature are called \DEF{local maximum (minimum) points}. Both of them are called \DEF{extremal points} of an HCMU surface. A smooth extremal point is an extremal point but not a conical or cusp singularity. 
     A saddle point of curvature function is call a \DEF{saddle point} of the HCMU surface.     
\end{definition}

We shall see the role of these points when the decomposition theorem is stated. On the other hand, they induce a stratification of moduli space discussed later. 

\bigskip
\subsection{The structure of HCMU surfaces}
As many other geometric objects, such as hyperbolic surfaces, HCMU surfaces can be built from smaller bricks. 

It is shown in \cite{CCW05} that an HCMU surface can be divided into a finite number of pieces, called football, by cutting along some geodesic segments. We will review this result in Section \ref{sec:football}. 
Such structure provides a topological method to construct HCMU surfaces. 

\medskip
Base on this football decomposition, we provide a detailed version of the structure of HCMU surfaces, which is the main tool in this work. 
\begin{theorem}\label{thm:main_data_rep}
     Most HCMU surface can be uniquely represented as a data set 
     \[ \left( \mathbb{A}, \PP^+ \sqcup \PP^-; K_0, R; \mathcal{W, L} \right). \]
     In the set, $\AA$ is a subdivision of the underlying surface into topological polygons; $\PP^+ \sqcup \PP^+$ is partition of the vertices of $\AA$ such that every edge in $\AA$ connects points in different part; $K_0, R$ are positive constants; $\mathcal{W, L}$ are positive functions on the edges and surfaces of the subdivision $\AA$. 
\end{theorem}

A complete version will be state as Theorem \ref{thm:hcmu_data} in Section \ref{sec:glue_data}, where precise definitions of all involved concepts are given. 
This representation is efficient. Every data set above will also recover a unique HCMU surface. 
Therefore, the study of HCMU surfaces can be streamlined by focusing on the relevant data of the underlying topological surfaces, 
allowing the analysis to be conducted independently.
We point out that not every HCMU surface can be efficiently represented. But this is enough for some basic properties of the moduli space, like the dimension count.  

\medskip
Let us talk a little more about these data. Roughly speaking, each edge in the subdivision $\AA$ represents a football in the decomposition. The vertices are the extremal points, while the faces represent saddle points. The graph $\mathcal{G}:=(\AA, \PP^+ \sqcup \PP^-)$ itself can be regarded as a dessin d'enfant on topological surface. 
The constant $K_0$ records the curvature at maximum points, and $R$ is a parameter essentially equivalent to the curvature at minimum points. We chose this rather than the curvature for convenience during the application of our representation. 
$\mathcal{W}$ records one of the cone angles in each football represented by edges of $\AA$. $\mathcal{L}$ records the position of saddles points when viewed from a football.

It is already recognized in \cite{CCW05} that ``the construction of an HCMU metric from football is a combination problem''. The first two data $(\AA, \PP^{+} \PP^{-})$ precisely describe the combinatorial structure involved. The remaining 4 data encode geometric parameters of the surface.

\medskip
The idea behind can be traced back to the study of \qdf s, initiated by Strebel \cite{Strebel} among many others. The subdivision used here is similar to the WKB triangulation used by Gaiotto-Moore-Neitzke \cite{GMN13adv} and Bridgeland-Smith \cite{BrSm15}. We will adopt a generalized version by Barbieri-M\"{o}ller-Qiu-So \cite{BMQS24} in Section \ref{sec:glue_data}. 
Said in a different way, the data $\AA$ and $\mathbb{P}^{+} \sqcup \mathbb{P}^{-}$ are topological skeleton of the character 1-form. Hence it is possible to study generic HCMU surfaces without referring the underlying Riemann surface.

\bigskip
\subsection{Moduli spaces and main results}\label{ssec:M_space}
Usually a moduli space is a space collecting all geometric objects of the same type. For HCMU surfaces, we are mostly interested in the moduli space with prescribed genus and assignment of cone angles, varying complex structure. The main results of this work are some primary global properties of such moduli space. 

Let $ \SYM{\vec{\alpha}} := (\alpha_1, \cdots, \alpha_n) \in \left( \mathbb{R}_{\geq 0} \setminus \{1\} \right) ^n$ be an \DEF{angle vector} to be prescribed. We are interested in finding all HCMU surfaces with given genus and conical singularities of cone angle $2\pi\alpha_1, \cdots, 2\pi\alpha_n$. 
Here $\vec{\alpha} \neq \vec{0}$ but we allow some of $\alpha_i$'s be zero, corresponding to cusp singularities. 

\medskip
\noindent\textbf{Convention.} For an angle vector $\vec{\alpha}$ as above, it is always assumed that the first $k$ components $\alpha_1, \cdots, \alpha_k$ are the only integers, and that the last $q$ components $\alpha_{n-q+1}, \cdots, \alpha_n$ are the only zeros (if there is any).

\newcommand{\Mhcmun}{\mathcal{M}hcmu_{g,n}(\vec{\alpha})}
\newcommand{\MhcmunT}{\mathcal{M}hcmu_{g,n}(\vec{\alpha};\vec{T})}
\newcommand{\Mhcmu}[2]{\mathcal{M}hcmu_{#1}(#2)}
\begin{definition}\label{defn:M_hcmu} 
    The (geometric) \DEF{moduli space} \SYM{$\Mhcmun$} is the set of isometry classes of HCMU surfaces of genus $g$ with $n$ conical or cusp singularities, whose angle vector is prescribed to be $\vec{\alpha}$. Here a cusp singularity is regarded as a cone point of zero angle. 
\end{definition}

Not all vectors are achievable by some HCMU surfaces. There are constraints on the angle vector, partially discussed in \cite{CCW05, Lzy07}. 
Our first result recovers the necessary and sufficient conditions to the existence of HCMU surfaces, with a uniform proof, friendly for visualization. 

\begin{theorem}[Angle constraints]\label{thm:angle_cnstr_hcmu}
    Given the genus $g\geq0$ and an angle vector $\vec{\alpha} \in \left( \mathbb{R}_{\geq 0} \setminus \{1\} \right) ^n$ as above, define $m_0, a_0$ as (the summation is zero when $k=0$)  
    \begin{align}
        m_0 &:= \sum_{i=1}^{k} \alpha_i - (2g-2+n) \ , \label{eq:m0} \\
        a_0 &:= \sum_{i=1}^{k} (\alpha_i-1) - (2g-2) \geq m_0 \ . \label{eq:a0}
    \end{align}
    Then $\Mhcmun$ is non-empty if and only if the following holds.
    \begin{enumerate}[(A). ]
        \item When none of $\alpha_i$ is zero, 
        \begin{enumerate}[(\Alph{enumi}.1). ]
            \item $a_0 \geq 3,\ m_0 \geq 0$, or
            \item $a_0 = 2,\ m_0 =1$, or
            \item $a_0 = 2,\ m_0 =0$ and $\alpha_{n-1} \neq \alpha_{n}$ ($k = n-2$ in this case). 
        \end{enumerate}
        \item When there are $q>0$ zeros among $\alpha_i$'s, 
        \begin{enumerate}[\quad \quad \quad]
            \item $a_0 \geq q+1,\ m_0 \geq 0$ .
        \end{enumerate} 
    \end{enumerate}
\end{theorem}

\begin{remark}\label{rmk:angle_cnstr_hcmu}
    $a_0 \geq 2$ is always required.

    \begin{enumerate}[(i). ]
        
        \item If $k=0$, then $a_0 = 2-2g \geq 2$ implies $g=0$ and $a_0=2$. Then $m_0 = -(2g-2+n) = 2-n \geq 0$ imples $n=1,2$. All surfaces in such cases are HCMU footballs. 

        \item If $g=0, k>0$, since $\alpha_i \geq 2$ for $1 \leq i \leq k$, we have $a_0 = \sum_{i=1}^{k} (\alpha_i-1) + 2 \geq 3$. Hence the only constraint is $m_0 \geq 0$. This coincides the results of \cite{CCW05}. 
        
        \item When $g>0, k>0$, the two inequalities are independent of each other. 
        For example, let $g\geq3$ and $\vec\alpha=(g,g)$. Then $m_0 = 0$ while $a_0 = 0$.
        On the other hand, let $\vec{\beta} = (g+1, g+2, \lambda, \lambda, \lambda, \lambda, \lambda, \lambda)$ for some $\lambda \notin \ZZ$. Then $a_0 = 3$ while $m_0 = -1$. 

        \item In the particular case where $g=0$ and $n=1$, $\vec{\alpha}$ is allowed to be $\vec{0}$. These are HCMU footballs with one cusp and a smooth maximum point.
    \end{enumerate}
\end{remark}

\medskip
Following the same idea for the existence theorem, our next result concerns HCMU surfaces with a single conical singularity of angle $2\pi\alpha$, where $\alpha\in\ZZ_{>1}$. 
If the conical singularity is an extremal point, then it is known to be an HCMU football. Hence we shall require it to be a saddle point. 
In such cases, there are extra smooth extremal points on the surface. We give an existence result based on the number of maximum and minimum points. Genus zero case is also obtained in \cite{MyjWzq24}. 

\begin{theorem}\label{thm:S_alpha}
    Given the genus $g\geq0$ and a positive integer $\alpha \geq 2g+2$, there exists a genus $g$ HCMU surface with a single conical singularity of angle $2\pi\alpha$ which is also a saddle point, together with $p$ smooth maximum points and $q$ smooth minimum points, if and only if the following holds: 
    \begin{enumerate}[(1). ] 
        \item $p>q>0$ and $p+q=\alpha-2g+1$;  
        \item when $g=0$, then 
        \begin{enumerate}[(A). ] 
            \item $q=1$, or 
            \item $q>1$ and $q \nmid p$.
        \end{enumerate}
        \item  when $g>0$, there is no extra condition.
    \end{enumerate}
\end{theorem}
The inequality $\alpha \geq 2g+2$ comes from the angle constraint in Theorem \ref{thm:angle_cnstr_hcmu}. The condition $p>q$ is basically required by the geometry of HCMU football. The remained ones are the genuine extra constraints. 
Roughly speaking, with the data set representation, we provide a graph theoretical explanation of why higher genus cases have less constraint and how the arithmetic condition emerges.

\medskip
Our final result is the dimension count of the moduli space with prescribed cone angle. 
These studies of moduli space provide a primary answers to a problem suggested by Q. Chen, X. X. Chen and Y. Y. Wu at the end of \cite{CCW05}.

\begin{theorem}[Dimension of HCMU moduli space]\label{thm:dim_hcmu}
    Let $\alpha$ be an angle vector, following the convention and conditions of Theorem \ref{thm:angle_cnstr_hcmu}. Recall that $k$ is the number of positive integers among $\alpha_i$'s.
    
    Then the real dimension $\textbf{dim}$ of $\Mhcmun$ is given by the follows.
    \begin{enumerate}[(A). ]
        \item If $k>0$ and there is no cusp, then $\textbf{dim}=2g+2k$. 
        \item If $k>0$ and there are $q>0$ cusps, then $\textbf{dim}=2g+2k+q-1$.
        \item If $k=0$, then $g=0,\ n=1$ or $2$, and $\textbf{dim}=1$. 
    \end{enumerate}
\end{theorem}

``Dimension'' here indicates the independent geometric parameters that determine the isometric classes of HCMU surfaces. The strategy is similar to the one for dihedral cone spherical surfaces in \cite{LuXu24}. 
We are not seriously talking about the topology of the moduli space. However, since every surface is built from smaller fundamental shapes, and we can completely parameterize each shape, one can think the topology to be induced by Gromov-Hausdorff convergence.

\bigskip
\subsection{Organization of this paper}
Section \ref{sec:football}, \ref{sec:glue_data} describe and develop the main concepts and tools used in this work. 
Each of Section \ref{sec:exist_thm}, \ref{sec:S_alpha}, \ref{sec:dim_hcmu} is devoted to one of the main theorems. Readers may turn to the beginning of each section for the idea and summary of the proof. 

\medskip
In Section \ref{sec:football}, we focus on how to build an HCMU surface from small bricks. We start by a review on HCMU football and its geometry, together with the football decomposition. Then we turn to our modified version of strip decomposition for further usage. 

In Section \ref{sec:glue_data}, we fully present our data set representation of HCMU surface. The key concept is mixed-angulation of surface, which is a subdivision, natural generalizing the triangulation and encoding the information of decomposition. All the other data are based on this mixed-angulation. 

In Section \ref{sec:exist_thm}, the refined moduli space is introduced, before the proof of existence theorem. Then the necessity and sufficiency part of Theorem \ref{thm:angle_cnstr_hcmu} is given separately. 

In section \ref{sec:S_alpha}, the proof of Theorem \ref{thm:S_alpha} is divided into genus zero case and positive genus case. The necessity part for genus zero case is given first. Then the sufficiency part of two cases are shown by construction respectively.

In the last section, we finish the dimension count. It is divided into four tasks. See the beginning of Section \ref{sec:dim_hcmu} for the outline of this part.

\bigskip
\section{The geometry of HCMU footballs and surfaces} \label{sec:football}
We begin the study of HCMU surfaces with footballs. The first two subsections review the geometry of HCMU football and the football decomposition. 
After a re-parametrization, the character line element is introduced as a brick of both footballs and surfaces. Then we state our modified decomposition where footballs are replaced by bigons. 
This section end up with a concrete example.  

\bigskip
\subsection{HCMU football}\label{ssec:football}
Footballs are the very first examples of HCMU surfaces, yet they are actually the fundamental shapes to build arbitrary HCMU surfaces. 
This subsection reviews the definition, the parameterization and basic properties of HCMU footballs in \cite[\S2.2]{CCW05}. 

\begin{definition}\label{defn:hcmu_football} 
    An \DEF{HCMU football} is a rotationally symmetric genus-zero HCMU surface $(M,\rho)$ with two conical singularities. 
    In fact, the singularities are the two extremal points of $K$.
\end{definition}

Let $p, q$ be the two conical singularities of angle $\alpha, \beta$ and $\alpha>\beta>0$. Let the induced distance $\mathrm{dist}_{\rho}(p,q)$ between the two points be $l >0$. 
The sphere is parameterized as $(u,\theta) \in [0,l]\times[0,2\pi]$, where $(u,0), (u,2\pi)$ are identified for any $u\in[0,l]$. The line $\{u=0\}$ corresponds to the cone point $p$, and $\{u=l\}$ corresponds to $q$. 

In this coordinate, the metric is expressed as
\begin{equation}
    \rho = \rmd u^2 + f^2(u) \rmd \theta^2 \ .
\end{equation} 
All the boundary conditions of $f(u)$ are given by
\begin{equation}\label{eq:hcmu_ftb}
    \renewcommand\arraystretch{1.2}
    \left\{ \begin{tabular}{l}
         $f(0)=f(l)=0$ , \\
         $f'(0)=\alpha, \ f'(l)=-\beta$ , \\
         $f(u)>0, \ \forall\ 0<u<l$ .
    \end{tabular} \right.
\end{equation}
The curvature function $K$, depending only on $u$-parameter, is shown to be decreasing monotonically from $p$ to $q$. If \SYM{$K_0:=K(0)$, $K_1:=K(l)$}, then $K$ satisfies the differential equation 
\begin{equation}\label{eq:core_diff_eq}
    {K'}^2 = -\frac13 (K-K_0)(K-K_1)(K+K_0+K_1) \ . 
\end{equation}
Let $A(\rho)$ be the total area of this HCMU football. 
Then $f$ and $K$ satisfy the following equations: 
\begin{equation}\label{eq:K_f_equation}
    K = - f'' / f ,\quad  K' = c \cdot f, \textrm{ where } c=\frac{2\pi}{A(\rho)}(K_1-K_0) \ .
\end{equation}
The cone angles follow the relations 
\begin{equation}\label{eq:angle_curvature}
    \renewcommand\arraystretch{1.5}
    \bigg\{\begin{tabular}{l}
         $\alpha = \frac{A(\rho)}{12\pi} (2K_0+K_1)$ \\
         $\beta  = \frac{A(\rho)}{12\pi} (K_0+2K_1)$ \\
    \end{tabular}
    , \quad
    \bigg\{\begin{tabular}{l}
         $K_0 = \frac{4\pi}{A(\rho)} (2\alpha-\beta)$ \\
         $K_1 = \frac{4\pi}{A(\rho)} (2\beta-\alpha)$ \\
    \end{tabular} \ .
\end{equation}
Therefore it is required that 
\begin{equation}
    K_0>0, \quad -\frac{K_0}{2} < K_1 < K_0 \ .
\end{equation}

\begin{theorem}[Theorem 2.7, \cite{CCW05}]
    If the area $A(g)$ and angles $\alpha, \beta$ at both extremal points are given, there exists a unique rotationally symmetric HCMU metric in $S^2$, which is an HCMU football.       
\end{theorem}

\medskip
We can also include the case of $K_1=-K_0/2$ corresponding to cusp singularities. 
Since $K'$ is not identically zero in \eqref{eq:core_diff_eq} when $K_1=-K_0/2$, one can directly solve the differential equation.

\begin{proposition}\label{prop:limit_cusp}
    When $K_1=-K_0/2$, the solution to differential equation 
    \eqref{eq:core_diff_eq} with the only boundary condition $ K(0) = K_0 >0 $ satisfies $\lim_{u\to+\infty} K(u) = -K_0/2$ .
\end{proposition}
\begin{proof}
    We simply solve the differential equation
    \begin{equation*}
        {K'(u)}^2 = -\frac13 (K-K_0)(K+\frac{K_0}{2})^2
    \end{equation*}
    without extra condition on the limit when $v\to+\infty$. 
    Since
    \begin{equation*}
        \rmd\left( - \frac{2 \mathrm{arctanh}\left( \sqrt{\frac{a-x}{a}} \right) }{\sqrt{a}} \right)
        = \frac{\rmd x}{ x \sqrt{a-x}} \ ,
    \end{equation*}
    we have 
    \begin{equation}\label{eq:sol_cusp}
        \mathrm{arctanh}\left( \sqrt{\frac{K_0-K(u)}{\frac32 K_0}} \right) = 
        \frac{\sqrt{K_0}}{2\sqrt{2}} u \ .
    \end{equation}
    Here 
    $$ \mathrm{arctanh}(x) = \frac12 \ln\left( \frac{1+x}{1-x} \right) , \quad -1<x<1 \ .$$
    We see $K(u) \to K_1=-K_0/2$ if and only if $u \to +\infty$. Hence $\lim_{u\to+\infty} K(u) = K_1$ is automatically satisfied when $K_1 = -K_0/2$. 
\end{proof}

\begin{lemma}\label{lem:limit_cusp_f}
    When $K_1=-K_0/2$, the asymptotic behavior of $f(u)$ is 
    \begin{equation}
        f(u), f'(u)=O(\ee^{-C_0 u})
    \end{equation}
    for some positive constant $C_0$ when $u\to+\infty$. 
    Hence the boundary condition $f(+\infty) = 0 $ in \eqref{eq:hcmu_ftb} is automatically satisfied, and $f'(+\infty) =0$. This verifies again that a cusp can be regarded as a cone point of zero angle. 
\end{lemma}
\begin{proof}
    Actually, let $C_0:=\sqrt{K_0/2}$ and $C_1:=\sqrt{3}C_0$, by \eqref{eq:sol_cusp} we have
    $$ \sqrt{K_0-K} = C_1 \frac{\ee^{C_0 u} -1}{\ee^{C_0 u} + 1} \ \longrightarrow\ C_1 \quad  (\textrm{ as } u\to+\infty\ ) \ . $$
    Then, together with $C_1^2= 3K_0/2$, 
    \begin{align*}
        K(u) + \frac{K_0}2 &= \frac23 K_0 - C_1^2 \left( \frac{\ee^{C_0 u}-1}{\ee^{C_0 u}+1} \right)^2 = C_1^2 \left( 1- (\frac{\ee^{C_0 u}-1}{\ee^{C_0 u}+1})^2 \right) \\
        &= \frac{4 C_1^2 \ee^{C_0 u}}{(\ee^{C_0 u}+1)^2} = O(\ee^{-C_0 u}) \quad (\textrm{ as } u\to+\infty\ ) \ .
    \end{align*}
    By \eqref{eq:K_f_equation} and \eqref{eq:angle_curvature}, $c = - C_1^4/6\alpha$ is always negative. Since $K' = c\cdot  f$, we have 
    $$ f(u) \ \sim\  \abs{K'(u)} \ \sim\  \left( K(u)+\frac{K_0}{2} \right) \sqrt{K_0-K(u)} = O(\ee^{-C_0 u}) \ . $$

    By a direct computation we know
    $$ K''(u) = \frac12 \left(\frac{K_0}{2} + K(u)\right) \left(\frac{K_0}{2} - K(u)\right) \ . $$
    Since $K_0/2 - K(u) \to K_0$ and $K(u) + K_0/2 \sim O(\ee^{-C_0 u})$ when $u\to+\infty$ as shown before, we have 
    $$ f'(u) \ \sim\  K''(u) = O(\ee^{-C_0 u}) \ . $$
\end{proof}

\begin{remark}
    It is possible to consider $K_1=K_0$ as well. In such case, the limit metric has constant positive curvature. Up to a renormalization, these extremal \Kahler\ metrics are cone spherical co-axial metrics. See, for example, \cite{CWWX15, Erm20, GT23, LuXu24} for more about these metrics. 
    Since we are considering non-CSC metrics on surface, the case of $K_1=K_0$ is excluded in this work. 
\end{remark}

\bigskip
\subsection{Football decomposition of HCMU surface}\label{ssec:football_decomp}
Now we provide a quick review on the football decomposition theorem in \cite{CCW05}. 

\begin{definition}
    An integral curve of the vector field $\vec{H}$ in Proposition \ref{prop:vector_fields} is called a \DEF{meridian}. 
    Hence a meridian is always a geodesic. 
\end{definition}
In the last subsection, a line with constant $\theta$-parameter is a meridian on an HCMU football. 
Meridians will serve as the cut locus for decomposition. 

\begin{lemma}[Lemma 3.1, \cite{CCW05}]\label{lem:glue_hcmu} 
    For $i=1,2$, let $(S_i,\rho_i)$ be an HCMU football of cone angle $\alpha_i>\beta_i>0$ and area $A(\rho_i)$. 
    Then they can be smoothly glued along two meridians or segments of two meridians, if and only if 
    \begin{equation}
        \frac{\alpha_1}{\alpha_2} = \frac{\beta_1}{\beta_2} = \frac{A(\rho_1)}{A(\rho_2)} \ .
    \end{equation}
Let $K_0^{(i)}, K_1^{(i)}$ be the maximum and minimum curvature on $(S_i,\rho_i)$. 
By \eqref{eq:angle_curvature}, the desired equation holds if and only if $K_0^{(1)}=K_0^{(2)}$ and $K_1^{(1)}=K_1^{(2)}$. 
\end{lemma}

This result suggests that, since we are not fixing the underlying Riemann surface, it seems more convenient to focus on the extremal value of curvature and regard the cone angles as some quantities transverse to meridians. 
This is the point which motivates us to introduce the concept of {\it character line element} next subsection. 

But before that, let us introduce the following theorem indicating that any HCMU surface can be built from footballs, on which 
this work is based.  

\begin{theorem}[Theorem 1.3, \cite{CCW05}] \label{thm:hcmu_decomp}
    Let $(M, \rho)$ be an HCMU surface, then there are finitely many meridian segments which connect extremal points and saddle points. 
    In fact, $M$ can be divided into a finite number of pieces by cutting along these segments where each piece is locally isometric to an HCMU football.
\end{theorem}

We will give an equivalent modified version soon, which is more friendly to combine ideas and methods from the study of meromorphic differentials.

\bigskip
\subsection{Character line element}
We present a re-parameterization of HCMU football in the spirit of measured foliation. 
This leads us to smaller elements to determine an HCMU football. 

Let $(u,\theta)$ as before in Section \ref{ssec:football}. Define 
$$ (v,\phi) := (u,\alpha \theta) \quad \textrm{ and } 
\quad \rho = \rmd v^2 + h^2(v) \rmd \phi ^2 \ . $$
Then $\rmd\phi = \alpha\cdot \rmd\theta$, and $h(v) = f(v)/\alpha $. \eqref{eq:hcmu_ftb} turns into
\begin{equation}\label{eq:hcmu_ftb_2}
    \renewcommand\arraystretch{1.2}
    \left\{ 
    \begin{tabular}{l}
         $h(0)=h(l)=0$, \\
         $h'(0)=1, \ h'(l)=-\beta/\alpha$, \\
         $h(u)>0, \ \forall\ 0<u<l$, \\
         $v\in[0,l],\ \phi\in[0,2\pi\alpha]$.
    \end{tabular}
    \right.
\end{equation}
This can be regarded as a kind of normalization, with the advantage of additivity: adding the cone angle corresponding to the union of intervals for $\phi$. And now the initial conditions for $h,h'$ are independent of the cone angle. The terminal value $h'(l)$ record the ratio of the two cone angles, which is more essential. See Figure \ref{fig:hcmu_football} as an accessible illustration.

Also note that, since $K=-f''/f$ in $(u,\theta)$ coordinates and $h=f/\alpha$, the differential equation for $K$ in $(v,\phi)$ coordinates coincides with \eqref{eq:core_diff_eq}. 

\begin{figure}
    \makebox[\textwidth][c]{\includegraphics[width=\linewidth]{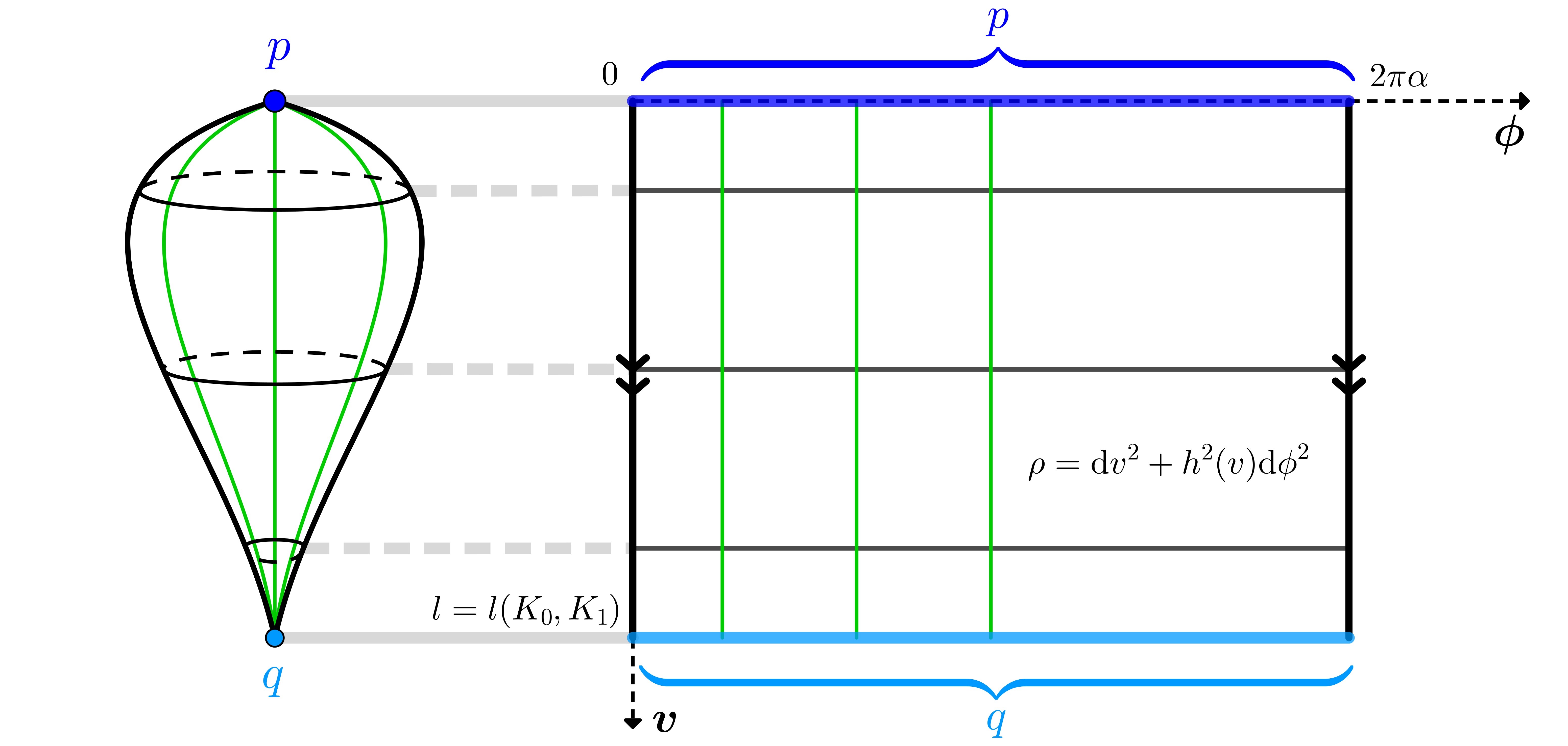}}
    \caption{\footnotesize A typical HCMU football with its $(v,\phi)$ coordinates. Note that $[0,l]\times\{0\}$ and $[0,l]\times\{2\pi\alpha\}$ are glued with respect to $\phi$-parameter. $\{0\}\times [0,2\pi\alpha]$ collapses to the conical singularity $p$, and $\{l\}\times [0,2\pi\alpha]$ collapses to $q$. Every green line is a meridian line as the character line element.}
    \label{fig:hcmu_football}
\end{figure}

\medskip
With this slightly modified parametrization, we introduce the character line element, which is the fundamental brick to build an HCMU football. 
\begin{definition}[Character line element]\label{defn:line_element}
    Given two constant $K_0, K_1$ satisfying 
    \begin{equation}\label{eq:K0K1_extd}
        K_0 >0 ,\quad K_0 > K_1 \geq -\frac{K_0}{2}  \ .
    \end{equation}
    An HCMU \DEF{character line element} of extremal curvatures $K_0, K_1$ is 
    a closed interval $[0,l] \subset [0,+\infty]$ endowed with a $C^2$-function $K(v)$, 
    satisfying the following: 
    \begin{enumerate}[(1). ]
        \item $K(v)$ is decreasing, with initial value $K(0)=K_0$.
        \item $K(v)$ satisfies the differential equation
        \begin{equation}
            3 {K'}^2 = -(K-K_0)(K-K_1)(K+K_0+K_1) \ .
        \end{equation}
        \item $l=l(K_0,K_1)\in[0,+\infty]$ is the unique solution of $K(v)=K_1$, called the \DEF{length} of this character line element.
    \end{enumerate}
    
    The \DEF{warped function} \SYM{$h(v)$} on the interval $[0,l]$ is defined to be 
    \begin{equation}\label{eq:density_h}
        h(v) := \frac{K'(v)}{\overline{c}}, \quad \textrm{ where }
        \overline{c} := -\frac16 (K_0-K_1)(2K_0+K_1) \ . 
    \end{equation} 
\end{definition}
Intuitively, an HCMU football is weaved by some width of character line element. Hence the character line element is the ``brick of bricks''. 

Also note that, by Proposition \ref{prop:limit_cusp} and Lemma \ref{lem:limit_cusp_f}, when $K_1=-K_0/2$, it is reasonable to define $l=+\infty$ and $h(+\infty)=h'(+\infty)=0$. 
All boundary conditions in \eqref{eq:hcmu_ftb_2} is still satisfied. So we allowing $K_1=-K_0/2$ in \eqref{eq:K0K1_extd}. 

\begin{remark}
    The name of ``warped function'' indicates that an HCMU football is actually a warped product of Riemannian manifolds. This concept can be found, for examples, in \cite[Chapter 7]{ONeill83} or \cite[Chapter 4]{GTM171}.  
\end{remark}

\newcommand{\Ratio}{R_{atio}(K_0, K_1)}
\begin{definition}[Ratio]\label{defn:ratio}
    The \DEF{ratio} of an HCMU character line element of extremal curvatures $K_0, K_1$ is defined to be 
    \begin{equation}\label{eq:Ratio} 
         \SYM{\Ratio} := \frac{K_0+2K_1}{2K_0+K_1} \ .
    \end{equation}
By the football decomposition Theorem \ref{thm:hcmu_decomp}, together with the second part of Lemma \ref{lem:glue_hcmu}, all character line elements weaving a football in the decomposition are the same. 
    Then the \DEF{ratio} $R$ of an HCMU surface is defined to be the ratio of this common character line element. 
\end{definition}

If the extremal curvatures $K_0, K_1$ of HCMU footballs are fixed, then by \eqref{eq:angle_curvature} the ratio of two cone angles $\beta, \alpha$ at the extremal point with curvature $K_1, K_0$ is always $\beta / \alpha = \Ratio$. 
This allows us to control the cone angles at minimum points by those at maximum points and the ration of surface. In other words, we shall only regard one side of extremal points as independent variables. We will adopt this idea all over the proofs of main theorems.

\bigskip
\subsection{Strip decomposition of HCMU surface}\label{ssec:hcmu_strip_decomp}
In our modified version of Theorem \ref{thm:hcmu_decomp}, the pieces of decomposition are bigons instead of footballs. 

\begin{definition}\label{defn:hcmu_bigon}
    An \DEF{HCMU bigon} of extremal curvatures $K_0, K_1$ and \DEF{top angle} $\alpha$, denoted by \SYM{$B_{\alpha}(K_0, K_1)$}, is a topological disk parameterized by 
    \[ (v, \phi) \in [0,l] \times [0, 2\pi\alpha] \]
    endowed with the metric 
    \[ \rho = \rmd v^2 + h^2(v) \rmd \phi^2 \ , \]
    where $h(v)$ is the warped function of character line element of extremal curvatures $K_0, K_1$ in Definition \ref{defn:line_element}. 
    Here $K_0,K_1$ satisfies \eqref{eq:K0K1_extd}. 

    As the footballs, the line $\{v=0\}$ collapses to one point on the boundary, called the \DEF{top vertex} of the bigon. And $\{v=l\}$ collapses to another point called \DEF{bottom vertex}, of \DEF{bottom angle} $\beta = \Ratio\cdot \alpha$. 
\end{definition}

The two boundaries of $B_{\alpha}(K_0, K_1)$, connecting the two vetices, are HCMU geodesics. If they are glued isometrically according to the $v$-parameter, one gets an HCMU football denoted by \SYM{$S_{\alpha}(K_0, K_1)$}. The cone points corresponding to $\{v=0\}$ and $\{v=l\}$ are also called the \DEF{top and bottom vertex} of the football, respectively, with the \DEF{top and bottom angles} defined similarly.  
In the literature like \cite{CCW05}, such football is denoted by $S^2_{\{\alpha, \beta\}}$ when $\beta / \alpha = \Ratio$. Compare Figure \ref{fig:hcmu_football} again.

\begin{corollary}[Strip decomposition for HCMU surfaces]\label{cor:hcmu_strip_decomp}
    Every HCMU surface can be decomposed into finitely many HCMU bigons by cutting along finite number of meridian segments. These segments are geodesics connecting critical points of curvature $K$. And all these HCMU bigons have the same extremal curvatures. 
\end{corollary}
This is the HCMU version of strip decomposition in \cite{LuXu24}, or the strip decomposition of meromorphic differentials (e.g. \cite{BrSm15}). 
The number of bigons, the division of the boundaries of the bigons, and the pairing of boundary segments are called the \DEF{gluing pattern} of the strip decomposition. This is a purely combinatorial data. We will study this data in details during Section \ref{sec:glue_data}. 

Conversely, we can obtain an HCMU surface by gluing several bigons of the same extremal curvatures. 
All the top (bottom) vertices of the bigons are glued to the maximum (minimum) points of the surface. 
A boundary point of a bigon, other than the two vertices, that is glued to a saddle cone point is considered as a \DEF{marked boundary point} of the closed bigon before gluing. 
A saddle point of cone angle $2\pi n\ (n\in\ZZ_{>1})$ is glued from $2n$ marked points. 

\medskip
Since an HCMU football is completely determined by its extremal curvatures and top angle, and the football or strip decomposition is canonical, we know that these pieces in the decomposition completely determine an HCMU surface.   
\begin{corollary}\label{cor:identification}
    Two HCMU surfaces are isometric if and only if they have the same strip decomposition and gluing pattern.
\end{corollary} 

\begin{figure}[th]
    \centering
    \subfigure[The strip decomposition into bigons.]{
    \includegraphics[width=0.8\textwidth]{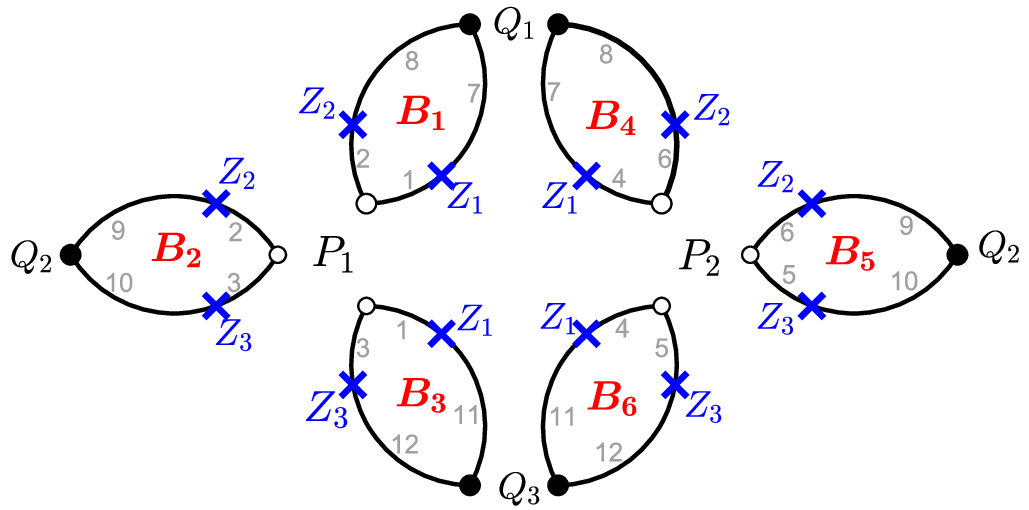}
    \label{fig:Calabi_1}
    }

    \subfigure[The strip decomposition drawn on the sphere.]{
    \includegraphics[width=0.8\textwidth]{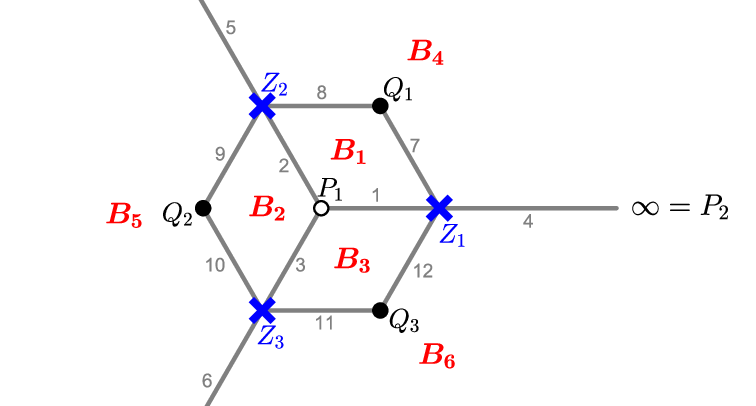}
    \label{fig:Calabi_2}
    }
    \caption{\footnotesize Calabi's example of an genus zero HCMU surface with 3 conical singularities of angle $4\pi$.}
    \label{fig:Calabi_pt1}
\end{figure}
We end this section with Calabi's example of HCMU surface. The precise description can be found in \cite[\S3.1]{Cxx00}. 
\begin{example}\label{eg:Calabi_1}
    Let $S_{\pi}(K_0, K_1)$ be an HCMU football of top angle $\pi$ and $K_1=K_0/4$. Then $\Ratio=\frac23$ and the bottom angle is $\frac23 \pi$. 
    Calabi's example is an HCMU surface in $\Mhcmu{0,3}{2,2,2}$ with nice symmetry. It admits a 6-sheet cover of $S_{\pi}(K_0, K_1)$, branched at the two cone points and another smooth point. There are 3 preimages of the top vertex and 2 preimages of the bottom vertex, all smooth. The three cone points of angle $4\pi$ are the preimage of the remained branched value, and are the saddle points of the HCMU surface. 
    Hence its strip decomposition consists of 6 identical bigons $B_{\pi}(K_0, K_1)$, with the marked boundary  points lies at exactly the same position. 
    
    Figure \ref{fig:Calabi_1} shows how to glue 6 HCMU bigons $B_{\pi}(K_0, K_1)$ to Calabi's example. The top vertex of a bigon is always denoted by a black vertex. The marked boundary points of each bigon are represented by cross symbols. The boundary segments labeled with the same number are glued in pair.  
    
    Figure \ref{fig:Calabi_2} draws the bigons and cut locus topologically on the original surface. It is regarded as a Riemann sphere, with the infinity a preimage of bottom vertex. $Z_1,Z_2,Z_3$ are the only conical singularities, all of angle $4\pi$. The label of edges matching the gluing data in Figure \ref{fig:Calabi_1}. 
\end{example}

\begin{remark}
    A complete classification of HCMU surfaces in $\Mhcmu{0,3}{2,2,2}$ is given by Wei-Wu \cite{WzqWyy19}.
\end{remark}

\bigskip
\section{Representing HCMU surfaces as data sets}\label{sec:glue_data}
This section shows how to faithfully represent most HCMU surfaces by a bunch of data, namely the full version of Theorem \ref{thm:main_data_rep} (see Theorem \ref{thm:hcmu_data}). 
There are both discrete topological data and continuous geometric parameters. 
Such method is inspired by the study of moduli space for meromorphic differentials. Our notations follow the generalized version by Barbieri-M\"{o}ller-Qiu-So in \cite{BMQS24}.

\bigskip
\subsection{Generic surface}\label{ssec:generic}
Every HCMU surface admits a unique strip decomposition. But the data set representation may not work efficiently on all the surfaces. 

\begin{definition}\label{defn:generic_hcmu}
    An HCMU surface $(M,g)$ is said to be \DEF{generic} if each boundary of the bigons in its strip decomposition contains exactly one marked point.
    This is equivalent to the absence of a meridian segment connecting two saddle points. 

    An HCMU football is appointed to be non-generic. 
\end{definition}

Basically, this is the same as what happens to most meromorphic differentials. Such topics about differentials are discussed in \cite{BrSm15, BMQS24}. 
In our case, the differentials are the character 1-forms with simple poles \cite{CqWyy11, CWX15}, whose structures are simpler. 

Here HCMU footballs are ruled out, because they do not contain any saddle point. Hence we do not talk about strip decomposition of a football, even though it is already glued from a bigon. 
In any case, HCMU footballs are simple enough themselves. Such convention brings convenience in later arguments. 

\medskip 
We will show that most HCMU surfaces are generic. That is, generic surfaces form a top-dimensional subset (possibly not connected), and non-generic ones always appear at the boundary. Hence they are indeede ``generic'' in the usual sense. 
This will be proved in Section \ref{ssec:twist}, which can be regarded as a simplified version of results on differentials mentioned above.

\bigskip
\subsection{Mixed-angulation as gluing data}\label{ssec:mixed_angult}
This subsection clarifies the concept of mixed-angulation. 
Most parts agree with \cite{BMQS24}. Adjustments to the names and conditions have been made to suit our specific usage. 

We start with the setting for topological surface. Also see \cite[\S 2]{FST08} or \cite[\S 3]{BMQS24}. We shall use a slightly different but simpler version, with empty boundary. 

\begin{definition}[Marked surface]\label{defn:mk_suf}    
A \DEF{marked surface} $\SYM{\mathbf{S}} := (S, \PP)$ consists of a connected smooth surface $S$ with a fixed orientation, together with a finite set $\SYM{\PP} := \{p_j\}_{j=1}^p$ of points regarded as \DEF{punctures}. 
\end{definition}

\noindent{\bf Convention.} If a bold capital letter $\mathbf{X}$ represents a marked surface, then the underlying topological surface is always denoted by the same unbolded capital letter $X$. And vice versa. 

\begin{definition}
    An \DEF{open arc} is (an isotopy class of) a curve $\gamma:[0,1]\to S$ such that its interior is in $S\setminus \PP$ and its endpoints are in $\PP$. 
    An (open) \DEF{arc system} $\SYM{\AA} := \{\gamma_i\}$ is a finite collection of open arcs on ${\bf S}$ such that each of them has no self-intersection in $S\setminus \PP$ and any two open arcs in $\AA$ have no intersection in $S\setminus \PP$. 
\end{definition}

\begin{definition}\label{defn:order_vec}
    An \DEF{order vector} $\SYM{{\bf w}} := (w_1,\cdots, w_r) $ is a positive 
    integer vector, up to permutations of the components. 
 An order vector ${\bf w}$ is said to be \DEF{compatible} with a marked surface ${\bf S}=(S,\mathbb{P})$ if 
    \[ \sum_{i=1}^r w_i = 4g-4 + 2|\mathbb{P}|=4g-4+2p \ . \]
\end{definition}
The order vector is called ``weight function'' on decorated points in \cite{BMQS24}. But we shall leave this name for the function on the edges later. 

\begin{definition}[Mixed-angulation]\label{defn:mix_angulation}
    Given a marked surface ${\bf S}$ and an order vector ${\bf w}$ compatible with ${\bf S}$. 
    A \DEF{${\bf w}$-mixed angulation} of ${\bf S}$ is an arcs system $\AA$ that divides $S$ into $r$ \DEF{complementary polygons} $\FF=\{F_1, \cdots, F_r\}$ such that $F_i$ is exactly a topological $(w_i+2)$-gon for all $1\leq i\leq r$. 
\end{definition}
The term ``$+2$'' is inherited from the setting for \qdf s: a zero of order $w$ in a \qdf\ is a cone of angle $(w+2)\pi$ in the corresponding flat surface.  

We are allowing self-folded objects in the setting. 
Hence a mixed-angulation is merely a cell-decomposition of the surface, not necessarily a simplicial complex decomposition. 
\begin{definition}\label{defn:self_fold}
    An open arc of a mixed-angulation is \DEF{self-folded} if it connects one puncture to a different one such that it is the only arc emitting from the latter puncture.
    A complementary polygon is \DEF{self-glued} if two of its edges are glued to the same open arc of the mixed-angulation.
    
    Equivalently, a self-folded arc is obtained by gluing two successive edges of one polygon. 
    And if the two sides of an arc belong to the same complementary polygon, then the polygon is self-glued. See Figure \ref{fig:poly_self_folded} as examples.
\end{definition}

\begin{figure}
    \makebox[\textwidth][c]{\includegraphics[width=0.61\textwidth]{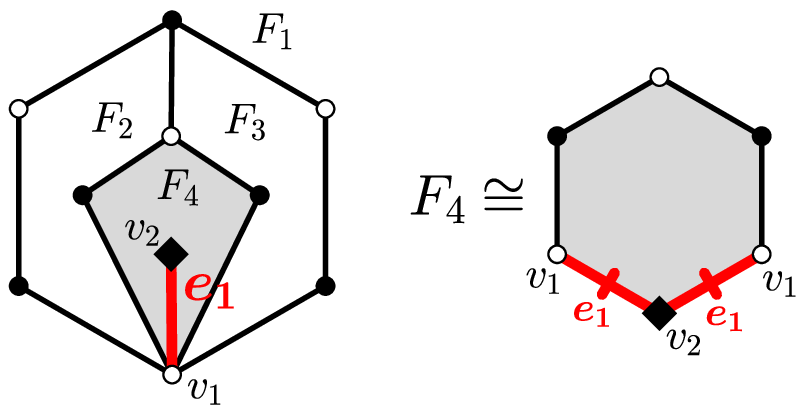}}
    \caption{\footnotesize A $(4,4,4,4)$-mixed angulation of sphere, with a self-folded arc $e_1$ and a self-glued polygon $F_4$.}
    \label{fig:poly_self_folded}
\end{figure}

\medskip \noindent
\textbf{Convention.} Two ${\bf w}$-mixed angulation $\AA_i$ on marked surface ${\bf S_i}=(S, \PP_i)\ (i=1,2)$, with the same underlying topological surface $S$, are said to be equivalent 
if there exits an \ortprehomeo\ $\Psi: S \to S$ such that $\Psi(\PP_1) = \PP_2$ as subsets and $\Psi(\AA_1) = \AA_2$ as homotopy classes relative to $\PP_2$. 
From now on a mixed-angulation will always refer to its equivalent class. 

This equivalence is introduced to cancel out the effect of relabeling on the surface. 
Since we are considering the moduli space of isometry class, rather than other space (like \teich\ space), the action of \homeo s should be regarded as trivial one.  We adopted this convention solely to simplify the description and will not delve deeper into topics such as the mapping class group.

\begin{proposition}
[Mixed-angulation induced from HCMU metric]\label{defn:AA_q} ~\\
    Let $(M,\rho)$ be a generic HCMU surface, with $r>0$ saddle points of cone angle $2\pi\cdot(\alpha_1, \cdots, \alpha_{j_0}),\ \alpha_j\in\ZZ_{>1}$. 
     Define an order vector ${\bf w}\in (2\ZZ_{>0})^{j_0}$ as 
     $$ w_j := 2\alpha_j -2 \ .$$ 
     There is a ${\bf w}$-mixed angulation $\SYM{\AA(\rho)}$ on $M$ induced from $\rho$ as below{\rm :} 
     \begin{itemize}
         \item The puncture set $\PP$ on the marked surface ${\bf M}:=(M, \PP)$ consists of all extremal points of the HCMU metric.  
         \item Let $B_1, \cdots, B_s$ be the pieces of its strip decomposition. Represent each $B_i$ by an open arc $e_i$ isotopic to an interior meridian of the bigon. Then define 
         $\AA(\rho):=\{e_i\}_{i=1}^s$. 
     \end{itemize}
\end{proposition} 
\begin{proof}    
We only need to show that the complementary polygons are prescribed by the cone angles of saddle points. 
Since $(M,\rho)$ is generic, each bigon in the strip decomposition contains exactly 2 marked points on its boundary. And our choice of the open arc representing this bigon separates the 2 marked points. 
Now a saddle point $z_j$ of cone angle $2\pi \alpha_j\ (\alpha_j\in\ZZ_{>1})$ is glued from $2\alpha_j$ marked points, each belonging to a half bigon (see the paragraphs after Corollary \ref{cor:hcmu_strip_decomp}). 
These $2\alpha_j$ half bigons are glued, along boundary segments, to a connected component of $M \setminus \AA(\rho)$. 
Then it is a topological $2\alpha_j$-gon, and the only saddle point inside it is $z_j$ itself. 
Hence we end up with a ${\bf w}$-mixed angulation, whose complementary polygons are 1-1 corresponding to the saddle points of $\rho$. 
See Figure \ref{fig:polya_example} as an example. Also see the dictionary in Table \ref{tab:dict_diff2aglt} later. 
\end{proof}

\begin{figure}
    \makebox[\textwidth][c]{\includegraphics[width=1.1\textwidth]{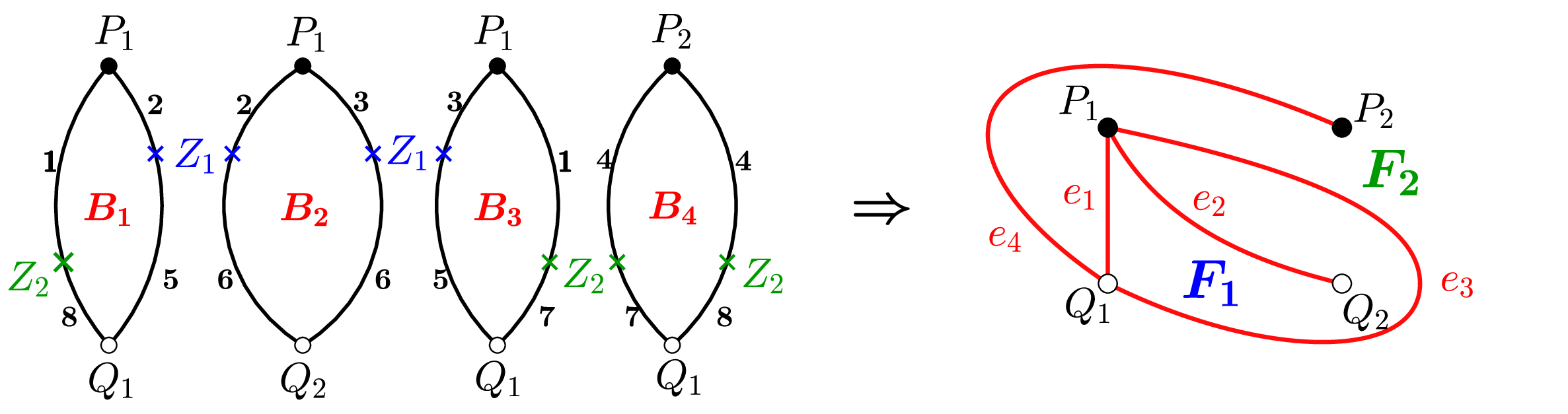}}
    \caption{\footnotesize From strip decomposition to mixed-angulation.  
    (Left). A strip decomposition of a genus 0 surface. Here $Z_1, Z_2$ are the saddle points of angle $4\pi$, and $P_1, P_2$ are two maximum points while $Q_1, Q_2$ are two minimum points. 
    (Right). The induced $(2,2)$-mixed angulation of the sphere. The infinity is contained in complementary polygon $F_2$. $e_2, e_4$ are self-folded edges.}
    \label{fig:polya_example}
\end{figure}

What happens to non-generic surfaces? By the same construction above in such surfaces, each complementary region may contains more than one saddle points. The number of boundary arcs in each region may not faithfully represent the individual cone angle of each saddle points inside it. 
It is still possible to achieve a finer presentation by adding additional data and refining subsequent methods, which we will not explore further in this work.

\bigskip
\subsection{Weighted bi-colored graph and balance equation}\label{ssec:weight_bicolor}

Now let $(M,\rho)$ be a generic HCMU surface. The puncture set $\PP$ and the induced mixed-angulation $\AA(\rho)$ will form a graph on the surface $M$. 
Let's introduce some common notations in graph theory first. 
\begin{definition} 
Let $\mathcal{G}=(V,E)$ be a simple graph without multi-edge or loop.
    \begin{itemize}
        \item If $v \in V$ is a vertex of edge $e \in E$, then we define \SYM{$v \in e$}.
        \item $\SYM{E(v)} := \defset{ e \in E }{ v \in e }$ is the set of all edges connecting to $v$. 
        \item The degree $\SYM{\deg (v)} := \abs{E(v)}$  
        of $v$ is the number of edges connecting to it. 
    \end{itemize} 
\end{definition}

\begin{definition}\label{defn:bicolor}
    A \DEF{bi-colored graph} $\mathcal{G}=(V,E)$ is a graph such that 
    \begin{itemize}
        \item the vertex set $V$ admits a partition $V^+\sqcup V^-$; 
        \item every edge $e\in E$ connects a point in $V^+$ and a point in $V^-$. 
    \end{itemize}
We shall regard vertices in $V^+$ as \DEF{black points}, and $V^-$ as \DEF{white points}. Then every edge in a bi-colored graph connects a black point to a white point.  
\end{definition} 
This is also called a bipartite graph in classical graph theory. We will use the notation $\mathcal{G}=(V^+\sqcup V^-, E)$ to emphasize the partition of vertices. 

\begin{definition}\label{defn:weight_graph}
    A \DEF{weighted graph} is a graph $\mathcal{G}=(V,E)$ together with a positive \DEF{weight function} $\mathcal{W}:E\to \RR_{>0}$ on the edge set. 
\end{definition}

\begin{proposition}\label{defn:weight_bicolor}
    The mixed-angulation $\AA(\rho)$ induced by a generic HCMU surface $(M,\rho)$ canonically corresponds to a weighted bi-colored graph $\mathcal{G}(\rho) := (\PP^+ \sqcup \PP^-, \AA(\rho); \mathcal{W})$. The weight function $\mathcal{W}: \AA(\rho)\to \RR_{>0}$ records the top angle of bigons in strip decomposition divided by $2\pi$.  
    In addition, $\mathcal{G}(\rho)$ is always connected. 
\end{proposition}
\begin{proof}
    Let $\PP^+, \PP^-$ be the set of maximum and minimum points of the HCMU surface respectively. Then every open arc in the mixed-angulation, regarded as an edge of the graph, connects a point in $\PP^+$ to a point in $\PP^-$. 
    If an open arc $e_i \in \AA(\rho)$ represents a bigon $B_i$ in the strip decomposition, then define $\mathcal{W}(e_i)$ to be the top angle of $B_i$ divided by $2\pi$. 

    Since $M$ is connected, for any $x_1\neq x_2 \in \PP$, there is a simple open arc $\gamma$ connecting them. We always assume the intersection of $\gamma$ and the cut locus of strip decomposition is finite. Applying local isotopy if necessary, we also assume $\gamma$ is disjoint from the saddle points. 
    Record the bigons that $\gamma$ passes through in order. Since $(M, \rho)$ is generic, two bigons $B_i, B_j$ appear as a adjacent pair in the sequence if and only if $\gamma$ cross a common boundary of them. Hence $e_i, e_j \in \AA(\rho)$ share a vertex in $\PP$, which is the extremal point on that common boundary segment. This induce a path from $x_1$ to $x_2$ in $\mathcal{G}(\rho)$. 
\end{proof}

Now let $\mathcal{W}$ be the weight function induced by a generic HCMU surface $(M,\rho)$ as above. Assume $x_i \in \PP^+$ is a maximum point of cone angle $2\pi\beta_i$. Here a smooth extremal point is always regarded as a cone point of angle $2\pi$. 
Then the sum of all top angles glued to $x_i$ must be $2\pi\beta_i$, according to the strip decomposition and the geometry of bigons. 
This is interpreted as the \DEF{balance equation} on the weight function at $x_i$:
\begin{equation}\label{eq:balance_black}
    \sum_{e\in E(x_i) } \mathcal{W}(e) = \beta_i \ .
\end{equation} 
Similarly, according to the geometry of bigons, if $R$ is the ratio of $(M,\rho)$ in Definition \ref{defn:ratio}, then the balance equation at a minimum point $x_j \in \PP^-$ is 
\begin{equation}\label{eq:balance_white}
    R\cdot \sum_{e\in E(x_j) } \mathcal{W}(e) = \beta_j \ .
\end{equation} 
The phrase ``balance equation'' comes from theory of network flows \cite{NetworkFlows}. One can regarded the top angles as the amount of cargo delivering from warehouses in $\PP^+$, the edges and weights as a transportation network, and $R$ as the damage rate. Then the balance equations simply count the amount of cargo each warehouse in $\PP^-$ will receive. 
In the problem of finding HCMU surfaces with prescribed cone angles, one equivalently needs to find suitable transportation networks fitting the distribution of cargo. 

\begin{remark}\label{rmk:ribbon_graph}
    It seems that the concept of mixed-angulation and bi-colored graph are somehow overlapping. 
    This is true to a certain extent, but the choice of concept depends on the content. 
    When talking about a mixed-angulation, there is always a background surface. The bi-colored graph is embedded into this surface, but can be described independently of the surface.  

    In later application, we need to construct a particular mixed-angulation on some surface. Since the topology of the surface is complicated, it would be easier to construct some graph on the plane first, then embed it into the surface. 
    So we shall switch between the mixed-angulation on surface and bi-colored graph on plane, due to the convenience of proof. 

    One alternative way to overcome this difference is using the concept of ``ribbon graph'', which is a key tool in the study of differentials. However, we are not going to that deep in this work. So it is not introduced here. ~\hfill $\square$
\end{remark}

\bigskip
\subsection{Data set representation of HCMU surfaces}\label{ssec:data_rep}
Now that we've encoded the gluing data of bigons as mixed-angulation of surface, and the top angles as a weighted function on the arcs of the mixed-angulation. 
The final data to determine is the position of the marked boundary  points. This will be encoded as the distance function. 
 
\begin{definition}\label{defn:face_funct}
    A \DEF{face function} of a mixed-angulation $\AA$ is a positive function $\mathcal{L}:\FF \to \RR_{>0}$ on the set $\FF$ of complementary polygons.  
\end{definition}

\begin{proposition}\label{defn:dist_funct}
    For a generic HCMU surface $(M,\rho)$, the length of meridian segments from saddle points to maximum points induces a well-defined face function $\mathcal{L}(\rho)$ of the mixed-angulation $\AA(\rho)$. 
    If $l$ is the length of the common character line element of $(M,\rho)$, then $0<\mathcal{L}(\rho)<l$. 
\end{proposition}
\begin{proof}
    Follow the setting in the proof of Proposition \ref{defn:AA_q}. The key point is that the length of boundary geodesics in every bigon all equal to $l$. 
    
    If $z_j$ is a saddle point of cone angle $2\pi\alpha_j,\ \alpha_j\in\ZZ_{>1}$, then the complementary polygon containing it is glued by $2\alpha_j$ half bigons. 
    The $4\alpha_j$ boundary segments of these half bigons glued to $2\alpha_j$ meridian segments in pairs, with a cyclic order surrounding $z_j$. 
    These segments, serving as the cut locus of the strip decomposition, end at maximal and minimum points alternately. 
    Any two adjacent segments in the cyclic order form a boundary of a bigon in the decomposition. 
    Since the length of all boundaries of these bigons are $l$, the length of any two adjacent segments in the cyclic order sums to $l$. 
    Hence the meridian segments emitting from $z_j$ and ending at maximal points all have the same length. 
    This is defined to be the distance $\mathcal{L}(\rho)(z_j)\in(0,l)$ from $z_j$ to the maximal points along the meridian. 
    
    Now that each complementary polygon corresponds to a unique saddle point, we obtain a well-defined face function $\mathcal{L}(\rho)$ of $\AA(\rho)$, controlling the position of saddle points. 
    Check Figure \ref{fig:Calabi_2} again for above reasoning.  
    (Note that all meridian segments from a fixed saddle points to minimum points also have the same length. But it will be infinity when minimum points are cusps, by Proposition \ref{prop:limit_cusp}. The distance to maximum points is always finite.)
\end{proof}

\medskip
Now we can state the complete version of Theorem \ref{thm:main_data_rep}.  
All the correspondence between objects on an HCMU surface and in the data set representation are summarized in Table \ref{tab:dict_diff2aglt}. 

\begin{theorem}[Data set representation of generic HCMU surface]\label{thm:hcmu_data}
    There is a 1-1 correspondence between all generic genus $g\geq0$ HCMU surfaces with conical or cusp singularities and the 
    following data sets on a genus $g$ smooth surface $S$: 
    \[ \left( \AA, \PP^+\sqcup \PP^-; K_0, R; \mathcal{W}, \mathcal{L} \right) \]
    where 
    \begin{enumerate}[1. ]
        \item $\AA$ is a ${\bf w}$-mixed angulation of a marked surface ${\bf S}=(S, \PP)$, with ${\bf w}
        \in (2\ZZ_{>0})^{j_0}$ for some positive integer $j_0$; 
        \item $\PP^+\sqcup \PP^-$ is a black and white coloring of $\PP$ such that $\mathcal{G}=(\PP^+\sqcup \PP^-, \AA)$ is a bi-colored graph; 

        \item $K_0>0$ and $0\leq R <1$ are two constants; 
        \item $\mathcal{W}: \AA \to \RR_+$ is a weight function on the set of arcs in the mixed-angulation; 
        \item $\mathcal{L}: \FF \to (0,l)$ is a face function on the set $\FF$ of complementary polygons of $\AA$, with $l=l(K_0,K_1)$ the length of character line element of extremal curvatures $K_0$, $K_1= \frac{2R-1}{2-R}\cdot K_0$. 
    \end{enumerate}
\end{theorem}

\begin{table}[t]
\centering
\renewcommand\arraystretch{1.25}
    \begin{tabular}{c|c}
    \hline
    \textbf{generic HCMU surface}    &    \textbf{data set representation} \\ \hline \hline
    maximum points    &    black points $\PP^+$ \\ \hline
    maximum points    &    white points $\PP^-$ \\ \hline
    bigons in strip decomposition    &   arcs in mixed-angulation $\AA$ \\ \hline
    a saddle point of angle $2\pi\alpha$    &    a complementary $(2\alpha)$-gon in $\FF$ \\ \hhline{=|=}
    top angles of bigons    &   weight function $\mathcal{W}:\AA \to \RR_{>0}$ \\ \hline
    position of saddle points     &    face function $\mathcal{L}:\FF \to (0,l)$ \\ \hline
    cone angle at extremal points     &    balance equations \\ \hline
    \multicolumn{2}{c}{ the ratio $R\in[0,1) $ } \\ \hline
    \multicolumn{2}{c}{ the maximum curvature $K_0 >0$ } \\ \hline
    \multicolumn{2}{c}{ the minimum curvature $K_1 = \frac{2R-1}{2-R}\cdot K_0$ } \\ \hline
    \multicolumn{2}{c}{ the length of character line element $l=l(K_0,K_1)$ } \\ \hline
    \end{tabular}
\vspace{5pt}
\caption{\footnotesize A dictionary between generic HCMU surface and its data set representation. The first 4 lines are discrete data, the middle 3 lines are continuous data, and the last 4 lines are continuous parameters in common. }
\label{tab:dict_diff2aglt}
\end{table}

\begin{proof}
    The procedure from an generic HCMU surface to a desired data set is the combination of Corollary \ref{cor:hcmu_strip_decomp}, Proposition \ref{defn:AA_q}, \ref{defn:weight_bicolor} and \ref{defn:dist_funct}. 

    \medskip
    The reverse procedure also holds. Let $\{e_i\}_{i=1}^{s}$ be the arc of mixed-angulation $\AA$, and $K_0, K_1, R, l$ be the constants given by the third and fifth conditions above. 
    First orient $\mathcal{G}$ so that every edge $e_i\in\AA$ start from a black point in $\PP^+$. This is feasible because $\mathcal{G}$ is bi-colored.
    
    The two functions $\mathcal{W, L}$ together with the constant $K_0, R$ determine the shape of all bigons to be glued, together with the positions of marked boundary  points. 
    For an arc $e_i\in \AA$, let $F_i^L, F_i^R \in \FF$ be the complementary polygons on the left and right hand side of $e_i$ with respect to the orientation 
    of $\mathcal{G}$. 
    Then construct the bigon $B_i$ to be the one of extremal curvatures $K_0, K_1$ and top angle $\mathcal{W}(e_i)$. Put it on the plane so that the top vertex is on the top. Then the marked point on the left (right) boundary of $B_i$ has the distance of $\mathcal{L}(F_i^L)$ ($\mathcal{L}(F_i^R)$) away from the top vertex. 

    Two bigons $B_i, B_j$ are glued along boundary segments connecting to their top vertices, if and only if $e_i, e_j$ are two successive boundary side of some polygon $F_k\in \FF$ with a common black vertex in $\PP^+$. When this happens, we may assume that $B_i$ is on the left of $F_k$ with respect to our given orientation. Hence $F_k$ is on the right side of $e_i$ and the left side of $e_j$. Then the upper right boundary segment of $B_i$ has the same length as the upper left boundary segment of $B_j$, all equal to $\mathcal{L}(F)$ by construction. They can be glued smoothly by Lemma \ref{lem:glue_hcmu}. 
    In particular, if $e_i$ is self-folded with a degree 1 black vertex $x\in\PP^+$, then the upper left and upper right boundary segment of $B_i$ are glued in pair. 
    The case for bottom vertices are similar. 

    By the half-bigon argument used in the proof of Proposition \ref{defn:AA_q}, the resulting surface is a genus $g$ HCMU surface whose induced mixed-angulation is precisely $\AA$. 
    A complementary polygon $F_i\in\FF$ corresponds to a saddle point of angle $(w_i+2)\pi$. There are $j_0$ saddle points in total. 
    
    All punctures in $\PP$ corresponds to extremal points. We temporarily regard a smooth extremal point as a cone point of angle $2\pi$. 
    For any $x_i\in\PP$, define $\alpha_i := \sum_{e\in E(x_i)} \mathcal{W}(e)$. 
    If $x_i\in\PP^+$ is a black point, then it corresponds to a maximum point of angle $2\pi\alpha_i$. 
    If $x_j\in\PP^-$ is a white point, then it corresponds to a minimum point of angle $2\pi R \alpha_j $.

    \medskip
    Finally, by Corollary \ref{cor:identification}, this correspondence is a bijection.  
\end{proof}

\begin{figure}[b]
    \subfigure[]{
    \includegraphics[width=0.8\textwidth]{fig2/Calabi_2.eps}
    }

    \subfigure[]{
    \includegraphics[width=0.75\textwidth]{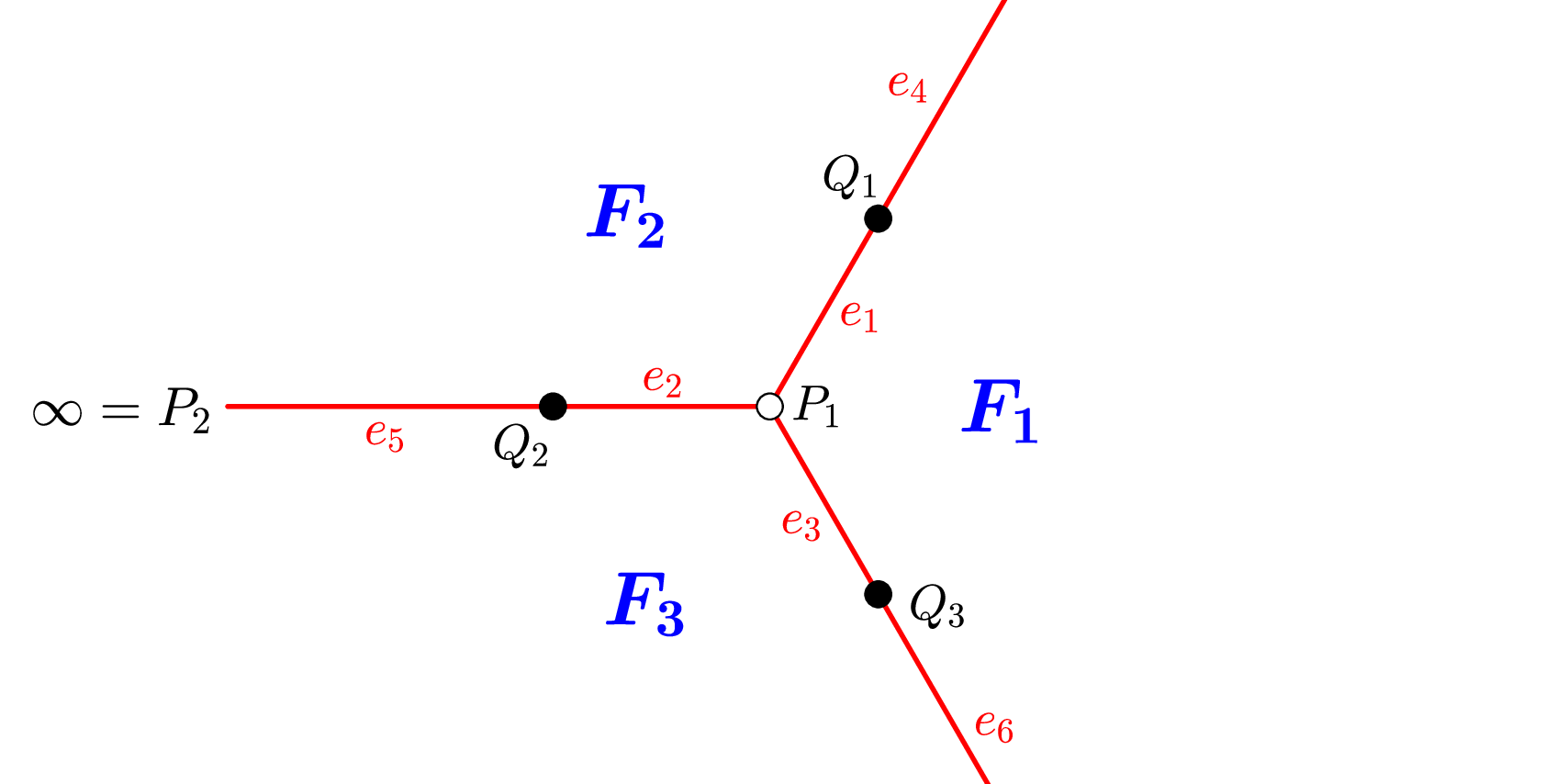}
    }
    \label{fig:Calabi_3_b}
    
    \caption{\footnotesize The mixed $(2,2,2)$-angulation of the sphere in Calabi's example.}
    \label{fig:Calabi_3}
\end{figure}

Let's conclude this section by checking Calabi's example again.    
\begin{example}  
    The induced mixed-angulation of Calabi's example in Figure \ref{fig:Calabi_pt1} is show in Figure \ref{fig:Calabi_3} (For the reader's convenience, Figure \ref{fig:Calabi_2} is copied to here). It is drawn as a bi-colored graph on sphere, with infinity being a white vertex.  

    As before, the ratio is $R=\frac23$. Recall that the strip decomposition consists of 6 identical bigons $B_{\pi}(K_0, K_1)$, with the marked boundary  points lies at exactly the same position. Then $\mathcal{W}$ is a constant weight function valuing at $\frac12$, and $\mathcal{L}$ is also a constant equal to the length of segment labeled by 7---12 (all the same). 

    The balance equation at each black vertex reads as $\frac12+\frac 12 = 1$. The equation at each white vertex reads as $\frac23(\frac12+\frac12+\frac12) = 1$. Hence all the 5 extremal points are actually smooth.
\end{example}

\newpage
\section{Existence theorem for HCMU surfaces}\label{sec:exist_thm} 
The first application of data set representation is the angle constraints for HCMU surfaces. This proof of Theorem \ref{thm:angle_cnstr_hcmu} has more topological and graph theoretical flavour. 

Before the proof of existence theorem, we will introduce a refined moduli space, where the type of each conical singularity is also prescribed. This is equivalent to fix the stratum of meromorphic differentials in which the character 1-forms lie. Then we state the refined version of angle constraints (Theorem \ref{thm:angle_cnstr_hcmu_new}). The original theorem is then a direct corollary.  

The necessity and sufficiency of the refined angle constraints are established in two subsections. The proof of necessity involves a detailed study on the number of extremal and saddle points. The proof of sufficiency is done by a graph theoretical construction on topological surfaces. 

\subsection{Type partition and refined moduli space}
Recall that every conical or cusp singularity is either a saddle point or an extremal point (Proposition \ref{prop:curvature}). 
When constructing an HCMU surface with prescribed angles, the cone angle of a saddle point must be an integer multiple of $2\pi$. On the other hand, the cone angle of a maximum or minimum point can be arbitrary positive number. 
Hence there is possibly a choice of type as which every conical singularity is realized. Equivalently, we are choosing whether each cone singularity is realized as a zero or a pole of the character 1-form. 
This leads to the following definition of refined moduli space. Our main theorems on moduli spaces are derived from the results on refined ones. 

\begin{definition}\label{defn:type_ptt_hcmu} 
    Given an angle vector $\vec{\alpha} \in \left( \mathbb{R}_{\geq 0} \setminus \{1\} \right) ^n $.
    A \DEF{type partition} of $\vec{\alpha}$ is a partition $\SYM{\vec{T}}=( Z, P^+, P^- )$ of $\{1,\cdots,n\}$ satisfying that
    \begin{itemize}
        \item $\forall i\in Z$, $\alpha_{i}\in \ZZ_{>1}$;
        \item if there exists some $\alpha_k=0$, then $P^- := \{ \ i\ \lvert\ \alpha_i=0\ \}$; 
        otherwise there is no restriction on $P^\pm$.
    \end{itemize}
\end{definition}

\begin{definition}[Refined moduli space] \label{defn:M_hcmu_T}
    \SYM{$\MhcmunT$} is the subset of $\Mhcmun$ consisting of isometry classes of surfaces such that
    \begin{itemize}
        \item $\forall i\in Z\ \ $, the $i$-th singularity $p_i$ is a saddle point of angle $2\pi\alpha_i$; 
        \item $\forall j\in P^+ $, the $j$-th singularity $p_j$ is a maximum point of angle $2\pi\alpha_j$; 
        \item $\forall k\in P^- $, the $k$-th singularity $p_k$ is a minimum point of angle $2\pi\alpha_k$; 
    \end{itemize}
\end{definition}

\noindent{\bf Convention.} 
When the components of $\vec{\alpha}$ share some same positive value, and if two type partitions are differed by a re-labeling of cone points with the same angle, we regarded them as the same (or equivalent) partition. 
For example, let $(\alpha_1,\alpha_2,\alpha_3,\alpha_4,\alpha_5) = (4, \frac12, \frac12, \frac13, \frac13)$. Then $(\{1\},\{2\},\{3,4,5\})$ and $(\{1\},\{3\},\{2,4,5\})$ are the same partition, but different from $(\{1\},\{4\},\{2,3,5\})$.

\medskip
Theorem \ref{thm:angle_cnstr_hcmu} is a direct corollary of the following refined version. 

\begin{theorem}[Refined angle constraints]\label{thm:angle_cnstr_hcmu_new}
    Always assume that $k>0$ and $Z\subset\{1,\cdots, k\}$ non-empty with $\abs{Z}=j_0>0$.
    Let
    \begin{align}
        m(Z) &:= \sum_{i\in Z} \alpha_i - (2g-2+n) \ , \label{eq:m(Z)} \\
        a(Z) &:= \sum_{i\in Z} (\alpha_i-1) - (2g-2) \geq m(Z) \ . \label{eq:a(Z)}
    \end{align}
    Then for a fixed $Z$, there exists a suitable type partition $\vec{T}=(Z,P^+,P^-)$ with non-empty $\MhcmunT$, if and only if the following conditions hold.
    \begin{enumerate}[(A). ]
        \item When none of $\alpha_i$ is zero, 
        \begin{enumerate}[(\Alph{enumi}.1).]
            \item $a(Z) \geq 3,\ m(Z) \geq 0$, or 
            \item $a(Z) = 2,\ m(Z) =1$, or
            \item $a(Z) = 2,\ m(Z) =0$ and $\alpha_{n-1} \neq \alpha_{n}$ (in this case $j_0 = k = n-2$). 
        \end{enumerate}
         In particular, when $\MhcmunT$ is non-empty, it always contains an HCMU surface with exactly 1 minimum point, i.e. $\abs{P^-} \leq 1$.
        \item When there are $q>0$ zeros among $\alpha_i$'s, 
            $a(Z) \geq q+1,\ m(Z) \geq 0$ .
    \end{enumerate}
   
\end{theorem}

\begin{remark}\label{rmk:j0_positive}
    As Remark \ref{rmk:angle_cnstr_hcmu}-(i), when $k=0$ or $j_0=0$, the moduli space must consist completely of HCMU footballs. 
    Hence we shall omit these cases and assume that $k \geq j_0 >0$ unless otherwise stated.    
\end{remark}

This theorem can be regarded as a re-statement of results in \cite{CCW05} and \cite{Lzy07}. In fact, it is proved by a similar idea, but with the language of mixed-angulation and weighted bi-colored graph. The process here is more intuitive and friendly for visualization, and also useful for the study of the whole moduli space. 

The angle constraints for arbitrary type partition seems to be challenging. Its difficulty will be partially revealed during our study of HCMU surface with one conical singularity in Section \ref{sec:S_alpha}.
    
\begin{table}[t]
\centering
\renewcommand\arraystretch{1.3}
\begin{tabular}{ccc|cl}
\multicolumn{3}{c|}{no cusp}                                                                                                      & \multicolumn{2}{c}{$q>0$ cusps}          \\ \hhline{===|==}
\multicolumn{1}{c|}{$a(Z) \geqslant 3$} & \multicolumn{2}{c|}{$a(Z)=2$}                                                          & \multicolumn{2}{c}{$a(Z) \geqslant q+1$} \\ \cline{1-3}
\multicolumn{1}{c|}{$m(Z) \geqslant 0$}  & \multicolumn{1}{c|}{$m(Z) = 1$}  & $m(Z) = 0$                                          & \multicolumn{2}{c}{$m(Z) \geqslant 0$}        \\
\multicolumn{1}{c|}{}                    & \multicolumn{1}{c|}{$n-j_0 = 1$} & $n-j_0 = 2$                                         & \multicolumn{2}{c}{}                     \\
\multicolumn{1}{l|}{}                    & \multicolumn{1}{l|}{}            & \multicolumn{1}{l|}{$\alpha_{n-1} \neq \alpha_{n}$} & \multicolumn{1}{l}{}          &         
\end{tabular}
\vspace{3pt}
\caption{\footnotesize A summary of Theorem \ref{thm:angle_cnstr_hcmu_new} (similar for Theorem \ref{thm:angle_cnstr_hcmu}).} 
\label{tab:angle_cnst}
\end{table}

\medskip
\begin{proof}[Proof of Theorem \ref{thm:angle_cnstr_hcmu} by Theorem \ref{thm:angle_cnstr_hcmu_new}] ~

    We only prove part (A). The case with cusps is similar, since every cusp must be a minimum point. 
    Obviously, $\Mhcmun$ is non-empty if and only if there exists some type partition $\vec{T}=(Z, P^+, P^-)$ with non-empty $\MhcmunT$. As mentioned in Remark \ref{rmk:j0_positive}, we only consider the case where $\abs{Z}>0$. 
    
    When $Z=\{1, \cdots, k\}$, $m(Z)=m_0,\ a(Z)=a_0$ and conditions in the two theorems are identical. 
    Hence the constraints in Theorem \ref{thm:angle_cnstr_hcmu}-(A) are always sufficient. 

    On the other hand, since $\alpha_i\geq2$ for all $1\leq i \leq k$, whenever $Z_1\subset Z_2 \subset \{1,\cdots,k\}$, $m(Z_1)\leq m(Z_2)$ and $a(Z_1)\leq a(Z_2)$. And if $Z_1\subsetneq Z_2$, the inequality is also strict. 
    Assume $\MhcmunT$ is non-empty for some type partition $\vec{T}=(Z, P^+, P^-)$. By the three cases in Theorem \ref{thm:angle_cnstr_hcmu_new}-(A), $a(Z)\geq2, m(Z)\geq 0$ always hold. 
    If $Z = \{1,\cdots,k\}$ then we already know conditions in the two theorems are the same.
    Otherwise $Z\subsetneq \{1,\cdots,k\}$, then $m_0>m(Z)\geq0,\ a_0 > a(Z) \geq 2$, and (A.1) of Theorem \ref{thm:angle_cnstr_hcmu} is satisfied. 
    Hence the constraints in Theorem \ref{thm:angle_cnstr_hcmu}-(A) are also necessary. 
\end{proof}

\bigskip
\subsection{The necessity of angle constraints}\label{ssec:necessity}
The proof here is an application of \PH\ formula for vector fields on surfaces. 

The \DEF{index} at a singularity of a vector field is defined as the total change in angle of the tangent vector when going around a small primitive loop around the singularity. 
The \DEF{\PH\ formula} asserts that, for a vector field on a smooth surface with finitely many singularities, the total sum of the index over all singularities equals to the Euler number. 
See \cite[Part Two, Chatper III]{Hopf83}, as an example, for the precise definition and proof of \PH\ formula. 

Suppose $(M,\rho) \in \MhcmunT$ with type partition $\vec{T}=(Z, P^+, P^-)$, $1 \leq \abs{Z} = j_0 \leq k$. 
Assume that there are $m^+\geq0$ and $m^-\geq0$ number of extra smooth maximum and minimum points on $(M,\rho)$. 

\begin{lemma}\label{lem:extrm_set}
    We have
    \begin{equation}\label{eq:PH_for_hcmu}
        \sum_{i\in Z} (1-\alpha_i) + (\abs{P^+}+m^+) + (\abs{P^-}+m^-) = 2-2g \ .
    \end{equation}
    Hence,  
    \begin{equation}\label{eq:extra_pole_hcmu}
        m = m(Z) := \sum_{i\in Z}\alpha_i - (2g-2+n) = m^+ + m^- 
    \end{equation}
    is exactly the number of smooth extremal points, and
    \begin{equation}\label{eq:pole_hcmu}
        a = a(Z) := \sum_{i\in Z}(\alpha_i -1) - (2g-2) = (\abs{P^+} + m^+) + (\abs{P^-} + m^-) 
    \end{equation}
    is the total number of all extremal points for every surface in $\MhcmunT$. 
\end{lemma}
\begin{proof}
Recall that $\vec{H}$ is the real gradient of curvature function $K$. As Proposition \ref{prop:curvature}, a singularities of $\vec{H}$ is either a saddle point or a extremal points of curvature function $K$. 
By the local behavior of its integral curve, the index at every extremal point is $+1$. The index at a saddle point of cone angle $2\pi\alpha\ (\alpha\in\ZZ_{>1})$ is $1-\alpha$. 
\end{proof}

For subsequent use, we denote by 
\begin{equation*}
    \SYM{p} := \abs{P^+} + m^+ \quad \textrm{and} \quad \SYM{q} := \abs{P^-} + m^-
\end{equation*} 
the total number of maximum and minimum points. Also let 
\begin{equation*}
    \SYM{A^+} := \sum_{i \in P^+} \alpha_i \quad \textrm{and} \quad \SYM{A^-} := \sum_{i \in P^-} \alpha_i 
\end{equation*}
be the sum of cone angles at maximum and minimum points.
All these quantities are summarized in Table \ref{tab:extrm_set}.

Note that when there are $q>0$ cusps, the minimum points are exactly the cusp singularities. Hence $m^-=0$ and $\abs{P^-}=q$ by definition.
So it is always reasonable to denote the total number of minimum points by $q$.

\begin{table}[t]
\centering
\renewcommand\arraystretch{1.3}
\begin{tabular}{l|c|c||c}
\multicolumn{1}{c|}{} & maximum point & minimun point & total number \\ \hline
conical point         & $\abs{P^+}$   & $\abs{P^-}$   & $n-j_0$      \\ \hline
smooth point          & $m^+$         & $m^-$         & $m=m(Z)$     \\ \hhline{=|=|=#=}
total number          & $p$           & $q$           & $a=a(Z)$     \\ \hline
total angle           & $A^+ + m^+$   & $A^- + m^-$   &             
\end{tabular}
\vspace{3pt}
\caption{\footnotesize A summary of settings for extremal points.}
\label{tab:extrm_set}
\end{table}

\begin{lemma}
    For an HCMU surface in $\MhcmunT$, its ratio is given by
    \begin{equation}\label{eq:ratio_given_type}
        R=\Ratio = \frac{\sum_{j\in P^-} \alpha_j + m^-}{\sum_{i\in P^+} \alpha_i + m^+} = \frac{A^- + m^-}{A^+ + m^+} \ .
    \end{equation}
\end{lemma}
\begin{proof}
    In Definition \ref{defn:ratio}, the ratio of an HCMU surface equals to the ratio of the common character line element. Hence it is also the ratio of the bottom and top angle of every HCMU bigon in its strip decomposition. 
    Now that each top (bottom) vertex is glued to a unique maximum (minimum) point of the surface, and the angle at the extremal points are given by the balance equation \eqref{eq:balance_black} or \eqref{eq:balance_white}. By summing up all these balance equation for maximum and minimum respectively, we have the desired equation \eqref{eq:ratio_given_type}. 
    Note that every HCMU surface contains at least one maximum point, $A^+ + m^+$ is always positive. 

    When there are cusp singularities, they form all the minimum points and are regarded as cone of zero angle. The ratio is identically zero in this case.  
\end{proof}

 \begin{lemma}\label{lem:number_b}
    For any generic $(M,g)\in\MhcmunT$, 
    the total number $b$ of bigons in its strip decomposition is 
    \begin{equation}\label{eq:b(Z)}
        \SYM{b = b(Z)} := n + m + (2g-2) = \sum_{i \in Z} \alpha_i \ .
    \end{equation}
\end{lemma}
\begin{proof}
    For a generic HCMU surface, each component of its strip decomposition can be regared as a topological quadrilateral. Since $M$ is an orientable surface, all the edges of the bigons as topological quadrilaterals are glued in pairs. Hence we can count the Euler number of $M$ by this cell decomposition as 
    \begin{equation}
        \chi(M) = 2-2g = (n+m) - (4b/2) + b = m+n-b \ . 
    \end{equation}
    The desired result is then obtained together with \eqref{eq:extra_pole_hcmu}.
\end{proof}

\begin{corollary}\label{cor:necessity}
    The conditions in 
    Theorem \ref{thm:angle_cnstr_hcmu_new} are necessary. 
\end{corollary}
\begin{proof} 
    Assume there is no cusp first. $m=m^+ + m^- \geq 0$ is by definition.
    There is at least 1 maximum and 1 minimum point on any HCMU surface. Hence the total number of all extremal points is $a=p+q \geq 1+1$. 

    In the rare case where $a=2$, there is exactly 1 maximum and 1 minimum point. We have $n-j_0+m=a=2$. Then $m\geq0$ implies $n-2 \leq j_0 \leq n$. 

    \begin{itemize}
        \item If $ j_0 = n-2$, then $m = m^+ = m^- =0$. Since $p\geq1, q\geq1$, we have $\abs{P^\pm}=1$. The ratio $R$ of an HCMU surface is strictly smaller than 1. By \eqref{eq:ratio_given_type}, we have  
        \[ R = \frac{A^- + m^-}{A^+ + m^+} = \frac{A^-}{A^+} < 1 . \]
        Here $(A^+, A^-)$ is an ordering of the two remained angles $\{\alpha_{n-1}, \alpha_{n}\}$. Then they must not be equal.

        \item If $ j_0 = n-1$, then $m=\abs{P^+}+\abs{P^-}=1$. One of the two extremal points is a cone point and the other a smooth point. Since the ratio is not 1, the remained angle $\alpha_n \neq 1$. 

        \item If $j_0 = n$, then $P^\pm = \varnothing$ and $m=2$. $p, q\geq1$ implies $m^+ = m^- = 1$. Then 
        \[ R = \frac{A^- + m^-}{A^+ + m^+} = \frac{m^-}{m^+} = 1, \]
        leading to a CSC surface. Such case is forbidden. 
    \end{itemize}

    \medskip
    Assume there are $q>0$ cusps next. 
    Then the minimum points are exactly the cusps. Hence $m^- = 0$ and $\abs{P^-}=q$. We have $a = p+q \geq 1+q$. 
\end{proof}

\bigskip
\subsection{The sufficiency of angle constraintt}\label{ssec:sufficiency}
The sufficiency part of Theorem \ref{thm:angle_cnstr_hcmu_new} requires the construction of HCMU surfaces in the refined moduli space $\MhcmunT$. Our strategy for
Case (A) of the theorem is to construct first a particular mixed-angulation with a single white point, then endowing the bi-colored graph with a suitable weight function. 
The resulting generic surface contains one minimum point, very similar to the ones constructed in \cite{CCW05, Lzy07}. Hence the following part can be regarded as a re-statement of these results, but with a different language. 


All the constructions are based on the following two elementary facts.

\begin{lemma}[Canonical mixed-angulation]\label{lem:canonical_polygon}
    On any closed surface $S_g$ of genus $g>0$, there is a canonical $(4g)$-mixed angulation $\AA_g^o$ such that $\abs{ \AA_g^o } = 2g+1$ and the induced bi-colored graph $\mathcal{G}_g^o$ has exactly 1 black and 1 white vertex. 

    When $g=0$, pick two different points and a simple arc connecting them. This can be regarded as a $(0)$-mixed angulation with an unique complementary bigon. So we shall allow $g=0$ in the above. 
\end{lemma}

For the next lemma, we shall also consider subdivision of polygons into topological polygons as mixed-angulation. The vertices of the polygon are regarded as puncutres on the boundary. Similar to Section \ref{ssec:weight_bicolor}, the vertices and arcs also form a bi-colored graph, but including those on the boundary. 

\begin{lemma}[Subdivision of a polygon]\label{lem:poly_poly}
    Let ${\bf w}\in (2\ZZ_{>0})^{j_0}$ be an even order vector and $K$ a positive integer such that 
    \[ 2L := \sum_{i=1}^{j_0} w_i - (2K-2) \geq 0 \ .\]
    Then a $2K$-gon admits a ${\bf w}$-mixed angulation endowed with a bi-colored graph structure, by adding $(j_0-1)$ diagonal arcs, $L$ self-folded arcs and $L$ black vertices. 
Here a \DEF{diagonal arc} means an interior open arc connecting a black and a white vertex.
The two vertices can be adjacent ones, and two arcs may share the same pair of vertices.
\end{lemma}

See Figure \ref{fig:canonical_poly} and Figure \ref{fig:polya_poly} as illustration for the two lemmas respectively. Basically, these figures explain the construction we are giving. 
The proof of the two lemmas are postponed after the proof of sufficiency of angle constraint.

\begin{figure}
    \setlength{\abovecaptionskip}{0pt}
    \includegraphics[width=0.7\textwidth]{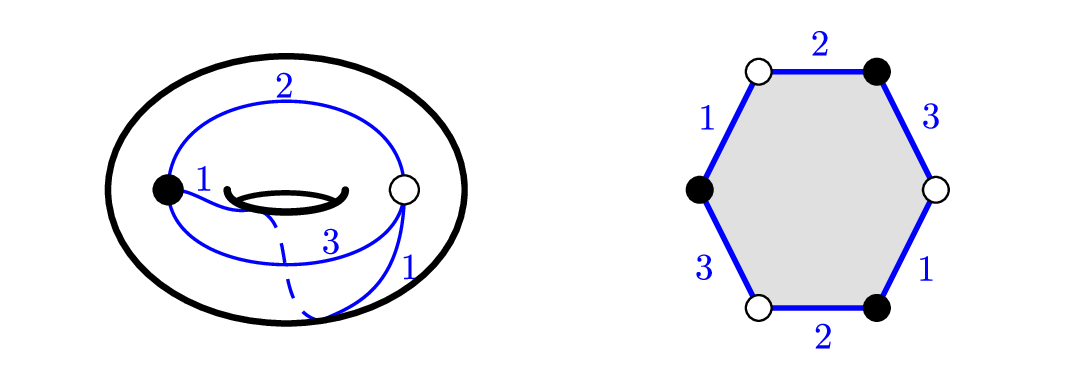}
    \caption{\footnotesize The canonical mixed-angulation of Lemma \ref{lem:canonical_polygon} for $g=1$ }
    \label{fig:canonical_poly}
\end{figure}
\begin{figure}
    \setlength{\abovecaptionskip}{2pt}
    \includegraphics[width=\textwidth]{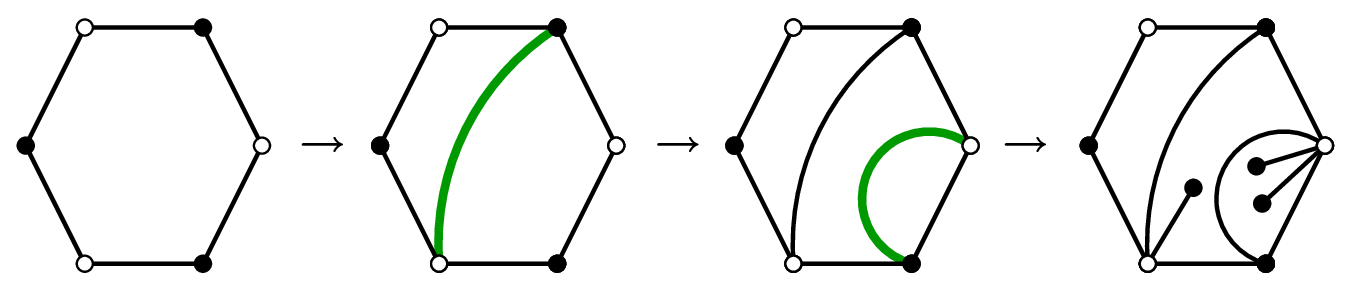}
    \caption{\footnotesize An example of Lemma \ref{lem:poly_poly} with $K=3, {\bf w}=(2,4,4)$. }
    \label{fig:polya_poly}
\end{figure}

\medskip
Some preparations are needed before the proof of sufficiency. Now assume $Z$ is fixed in the type partition. 
Recall that in Proposition \ref{defn:AA_q}, the relation between the angle vector and order vector is given by
$$ w_j = 2 \alpha_j - 2 \ . $$
Define $\vec{\beta} := \left( ( \alpha_k )_{k \notin Z}, 1, \cdots, 1 \right)$ be the \DEF{expected angle vector}, consisting of all the angles which are expected to be realized as extremal points, together with $m$ instances of 1 representing smooth extremal points. We shall regard them as cone points of angle $2 \pi$, for convenience.

\begin{proof}[Proof of Theorem \ref{thm:angle_cnstr_hcmu_new}.(A)] 
    To find the suitable choice of type partition $\vec{T}$ with fixed $Z$, and to construct an HCMU surface in $\MhcmunT$, we construct the corresponding data set in 2 steps.

    In The first step, we construct a ${\bf w}$-mixed angulation with a single white vertex in the coloring of its puncture set. 

    In the next step, we properly assign the expected angles to the punctures, and show that such an assignment always admits a weight function. The ratio is uniquely determined.

    Finally, the choice maximum curvature $K_0$ and the face function $\mathcal{L}$ is arbitrary and do not affect the existence. Hence the 4 elements in the data set constructed above is enough.
    
    \bigskip
    \textbf{(Step 1).} The construction of $(\AA, \PP^+ \sqcup \PP^-)$. 

    By Lemma \ref{lem:extrm_set}, Lemma \ref{lem:number_b}, Table \ref{tab:extrm_set} and Proposition \ref{defn:AA_q}, we are going to construct a ${\bf w}$-mixed angulation with
    \begin{itemize}
        \item $a= \sum_{i\in Z} (\alpha_i-1) - (2g-2)$ punctures,
        \item $b= n+m+2g-2 =\sum_{i\in Z} \alpha_i$ arcs,
        \item $\abs{Z} =j_0$ complementary polygons,
        \item a black and white coloring $\PP^+ \sqcup \PP^-$ of punctures. 
    \end{itemize}
    In fact, we will set $\abs{\PP^-}=\abs{P^-}+ m^- =1$ and $\abs{\PP^+} = \abs{P^+} + m^+ = m-1$.

    Let $\mathbb{A}_g^o$ be the canonical $(4g)$-mixed angulation in Lemma \ref{lem:canonical_polygon} and $\mathbb{P}^{-}$ be the white vertex. Its unique complementary polygon can be regarded as a bi-colored cycle of length $4g+2$ on the plane. Next we add more arcs and black vertices inside this polygon.
    Since 
    $$ 2L :\triangleq \sum_{i \in Z}{w_i} - 4g = 2 \sum_{i \in Z}\left(\alpha_{i}-1\right) -4g = 2a+(4g-4)-4g = 2a-4 \geq 0 \ ,$$
    we can apply Lemma \ref{lem:poly_poly} to the $(4g+2)$-gon obtained above. 
    We obtain a ${\bf w}$-mixed angulation by adding 
    \begin{itemize}
        \item $j_0 -1$ diagonal arcs,
        \item $L=\sum_{i \in Z}\left(\alpha_i-1\right)-2g =\sum_{i \in Z} \alpha_i-\left(2g+j_0\right)$ self-folded arcs,
        \item $L=a-2$ black vertices.
    \end{itemize}
    Drawing back the subdivision of this polygon on the surface, we have
    \begin{itemize}
        \item $(a-2)+1=a-1$ black vertices,
        \item $1$ white vertex,
        \item $(2g+1)+\left(j_0-1\right)+L = (2g+j_0) + \left(\sum_{i \in Z} \alpha_i-(2g+j_0)\right) = \sum_{i \in Z} \alpha_i$ arcs.
    \end{itemize}
    This is the desired ${\bf w}$-mixed angulation on $S$ indeed.

    \bigskip
    \textbf{(Step 2).} The construction of $R, \mathcal{W}$.  

    Let $\mathbb{P}^{+}=\left\{x_1, \cdots, x_a \right\}$ and $\mathbb{P}^-=\left\{x_0\right\}$. 
    We are going to assign the components of the expected angle vector
    $\vec{\beta} = \left( ( \alpha_k )_{k \notin Z}, 1, \cdots, 1 \right)$ 
    to $\PP$. The point is always assign the smallest expected angle to the unique white vertex.
    
    Let $\beta_{\min} := \min_{1 \leq k \leq a}\left\{\beta_k\right\}$ be the smallest component of $\vec{\beta}$.
    If such assignment admits a weight function, then similar to \eqref{eq:ratio_given_type}, the expected ratio of the constructed surface should be 
    $$ R=\frac{\beta_{\min}}{\sum_{i=1}^a \beta_i-\beta_{\min}} \ .$$

    \medskip
    \textbf{Case (A.1).} $a\geq3,\ m\geq0$. Then $\abs{\PP} = a \geq 3$.
    
    $\sum_{i=1}^a \beta_i \geq 3 \beta_{\min}$. Hence $ 0<R \leq \frac{1}{2} $. 
    Assign the remained components of $\vec{\beta}$ to black vertices in arbitrary order. Denote by $\beta_{\tau(i)}$ the expected angle assigned to $x_i \in \mathbb{P^+}$. 
    There is an obvious solution to the balance equation: 
    for any $x_i \in \mathbb{P}^{+}$ and $e \in E\left(x_i\right) \subset \mathbb{A}$, define
    $$ \mathcal{W}(e)=\frac{\beta_{\tau (i)}}{\deg\left(x_i\right)} \ . $$
    Then the total sum of the balance equation at $x_i$ is directly $\beta_{\tau (i)}$. Since every $e \in \mathbb{A}$ ends at the unique white vertex $x_0$, the balance equation at $x_0$ reads as
    \begin{align*}
    R \cdot \sum_{e \in \mathbb{A}} \mathcal{W}(e) & =R \cdot \sum_{x_i \in \mathbb{P}^{+}} \left(\sum_{e \in E\left(x_i\right)} \mathcal{W}(e) \right) =R \cdot \sum_{x_i \in \mathbb{P}^{+}} \beta_{\tau(i)} \\
    & =R \cdot\left(\sum_{k=1}^a \beta_k-\beta_{\min}\right)=\beta_{\text {min }} .
    \end{align*}
    We result in a surface realizing the expected angle vector at the extremal points. $P^\pm$ in the type partition is obtained by ruling out the smooth points in $\mathbb{P}^\pm$ according to our assignment of expected angle. Hence $\abs{P^-} \leq \abs{\mathbb{P}^-} = 1$. 
    
    \medskip
    \textbf{Case (A.2).} $a=2,\ m=1$. 

    When $a=2$, $L=a-2=0$ and no extra black vertex is added. By our convention on $\vec{\alpha}$, we have $\vec{\beta}=\left(\alpha_n, 1\right),\ \alpha_n \neq 1$. Then $R=\frac{\beta_{\min }}{\beta_{\max}}<1$. We shall assign $\beta_{\min}$ to $x_0$ and $\beta_{\max}$ to $x_1$.
    An obvious weight function is a constant function 
    $$ \mathcal{W}\equiv\frac{\beta_{\max }}{2 g+1} \ . $$
    Since every $e\in\AA$ is an arc connecting $x_0$ to $x_1$, the balance equation \eqref{eq:balance_white} at $x_0$ automatically holds.

    \medskip
    \textbf{Case (A.3).} $a=2,\ m=0$ and $\alpha_{n-1}\neq\alpha_n$.

    As previous case, $\vec{\beta}=\left(\alpha_{n-1}, \alpha_n\right)$, and we assign $\beta_{\max}$ to $x_1$ and $\beta_{\min }$ to $x_0$. Then $R=\frac{\beta_{\min}}{\beta_{\max}}<1$ by the assumption $\alpha_{n-1}\neq\alpha_n$. And a similar constant weight function
    $$ \mathcal{W}\equiv\frac{\beta_{\max }}{2 g+1} \ . $$
    satisfies the requirement.
\end{proof}

\begin{proof}[Proof of Theorem \ref{thm:angle_cnstr_hcmu_new}.(B)] 
    The case with cusps are similar. One tiny adjustment in previous (Step 1) would make it work again.

    By the assumption, $\abs{\mathbb{P}}=a \geq q+1$. Since the cusps would be the only minimum points, we shall construct a ${\bf w}$-mixed angulation with
    \begin{itemize}
        \item $q>0$ white vertices,
        \item $a-q = \sum_{i\in Z}\left(\alpha_i-1\right)-(2g-2)-q$ black vertices, 
        \item $b=n+m+2g-2=\sum_{i \in Z} \alpha_i$ arcs . 
    \end{itemize}
    $\mathbb{A}_g^o$ is obtained by Lemma \ref{lem:canonical_polygon} again. But before applying Lemma \ref{lem:poly_poly}, we need a small modification. 
    Add $q-1 \geq 0$ different white vertices inside the complementary $(4g+2)$-gon of $\mathbb{A}_g^o$, and choose $(q-1)$ disjoint self-folded arcs connecting each of them to some black vertex inside the $(4g+2)$-gon. 
    We obtain a topological $(4g+2)+2(q-1)=\left(4g+2q\right)$-gon.
    Now that 
    \begin{align*}
        2L &:= \sum_{i\in Z} w_i - (4g+2q-2) \\
        &= (2 a+4g-4) - (4g+2q-2) = 2(a-q-1) \geq 0
    \end{align*}
    by the assumption $a \geq q+1$. Hence we can apply Lemma \ref{lem:poly_poly} in this $(4g+2)$-gon to obtain a ${\bf w}$-mixed angulation of it.

    Drawing this back on the surface, we have
    \begin{itemize}
        \item $(q-1)+1=q$ white vertices,
        \item $1+L=1+(a-q-1)=a-q$ black vertices,
        \item $(2g+1)+(q-1)+(j_0-1)+L = 
        \sum_{i\in Z}(\alpha_i -1)$ arcs.
    \end{itemize}
    This is the desired ${\bf w}$-mixed angulation again.

    The weight function is even simpler to construct. Since $R\equiv0$ in such case, the balance equation \eqref{eq:balance_white} at each white vertices vanishes. 
    For an arbitrary assignment of expected angle, let $\beta_{\tau(i)}$ as before. Then for any $x_i \in \mathbb{P}^{+}$ and $e\in E(x_i)$, 
    $$ \mathcal{W}(e):=\frac{\beta_{\tau(i)}}{\deg(x_i)} $$
    is a solution, for the same reason as Case (A.1).
\end{proof}

The proof of the two lemmas are elementary but interesting. 

\begin{proof}[Proof of Lemma \ref{lem:canonical_polygon}]
    Take a regular $(4g+2)$-gon on the Euclidean plane, and pairing the opposite parallel sides by translations. Then all its vertices are divided into two quotient classes, and any two adjacent vertices belong to different class. 
    By computing its Euler number, the result surface is of genus $g$. 
    Take one of the quotient classes to be the black vertex on surface, and the other be the white vertex, we get the desired $(4g)$-mixed angulation. 
\end{proof}

\begin{proof}[Proof of Lemma \ref{lem:poly_poly}]
    This is done by an induction on $j_0$. 

    If $j_0=1$. Then no diagonal arc is needed. Add directly $L=\frac{1}{2} w_1-(K-1)$ different black vertices inside the polygon, and connect each to one of the white vertex by $L$ disjoint self-folded arcs.

    Now assume the Lemma holds for order vectors of length $j_0-1 \geq 1$. And ${\bf w}$ is another order vector of length $j_0$, satisfying the conditions for some $K \in \ZZ_{>0}$, i.e. $2L := \sum w_i-(2 K-2) \geq 0$.

    If $w_1=2z_1-2<2K-2$, then we pick arbitrary $(w_1+1)$ successive edges of the $(2K)$-gon, and add a diagonal arc cutting of these edges. Since $w_1 \in \mathbb{Z}_{>1}$. this diagonal arc must connect a black and a white vertex. The result is a $(w_1+2)$-gon and a $(2K-w_1)$ goon. Now that the remained $(j_0-1)$ components of order vector satisfies
    $$ \left(\sum_{i=2}^{j_0} w_i\right) - \left(2K - w_1 - 2\right) = \sum_{i=1}^{j_0} w_i - (2 K-2) \geq0 \ . $$
    Hence we can apply the induction hypothesis on the $(2K-w_1)$-gon.

    Otherwise, if $w_1+2 \geq 2K$, we simply pick an arbitrary diagonal arc connecting two adjacent vertices, dividing the polygon into a $(2K)$-gon and a bigon.
    For the $(2K)$-gon, we add $\frac{1}{2}\left(w_1+2-2 K\right)$ black vertices and self-folded arcs as before, to obtain a topological $(w_1+2)$-gon.
    For the bigon, we have $2L' := \sum_{i=2}^{j_0} w_i-(2-2) \geq 2$. By induction hypothesis, we add $(j_0-2)$ diagonal arcs, $L'$ black vertices and $L'$ self-folded arcs to make it a $(w_2, \cdots, w_{j_0})$-mixed angulation.

    Putting them together, we have a ${\bf w}$-mixed angulation, with extra
    \begin{itemize}
        \item $\frac12(w_1+2-2K) + L' = \frac12\left( \sum w_i - (2K-2) \right) = L$ black vertices,
        \item $1+(j_0-2) =j_0-1$ diagonal arcs,
        \item $\frac12(w_1+2-2K) + L' = L$ self-folded arcs.
    \end{itemize}
    This is the desired mixed-angulation. 
\end{proof}

\bigskip
\section{HCMU surfaces with one conical singularity}\label{sec:S_alpha}
This section is devoted to Theorem \ref{thm:S_alpha}, namely the classification of HCMU surfaces with a unique cone, which is also a saddle point. In other words, we are considering $\vec{\alpha}=(\alpha) \in \mathbb{Z}_{>1}$ and type partition $\vec{T}=(\{1\}, \varnothing, \varnothing)$.

As pointed out in Remark \ref{rmk:angle_cnstr_hcmu}-(i), the HCMU surfaces without saddle point must be footballs. 
So when $g>0$, the unique conical singularity must be a saddle point. Hence $\Mhcmu{g,1}{\alpha} = \Mhcmu{g,1}{(\alpha);\vec{T}}$. 
Together with football cases, Theorem \ref{thm:S_alpha} can be regarded as a coarse classification of all HCMU surfaces with one conical singularity. 

Before concrete construction of surfaces with one conical singularity, we show that the first condition in Theorem \ref{thm:S_alpha} is a general requirement. 
So In the follows subsections, these conditions are always assumed.
\begin{lemma}\label{lem:S_alpha.1}
    The conditions $\alpha \geq 2g+2$ and $p>q>0$, $p+q=\alpha-2g+1$ in Theorem \ref{thm:S_alpha} are necessary in all cases.
\end{lemma}
\begin{proof}
    Since the only conical singularity is a saddle point, every extremal point is a smooth one, that is, $a=m$. 
    Hence $a=m=2$ is not allowed by Theorem \ref{thm:angle_cnstr_hcmu_new}, we have $a=(\alpha-1)-(2g-2) \geq 3$. This implies $\alpha \geq 2g+2$, and $\Mhcmu{g,1}{(\alpha)}$ is always non-empty. 
    
    By \eqref{eq:ratio_given_type}, the expected ratio of the surface should be 
    $R =\frac{m^-}{m^+} =\frac{q}{p}$.
    Hence $q<p$ is always required. 
    Finally,
    $$ p+q=a=m=\alpha-(2g-2+1)=\alpha-2g+1 \ . $$
\end{proof}

\begin{remark}
    By the proof in this section, we also see why there is less constraints in positive genus cases. 
    When $g=0$, all the induced bi-colored graph is actually a planar graph. So positive genus allows more complexity for the bi-colored graph, bringing more possibility for the existence of weight function.
\end{remark}

\medskip
To construct a surface in $\Mhcmu{g,1}{(\alpha),\vec{T}}$ with $p$ maximum points and $q$ minimum point, we need to construct a $(2\alpha-2)$-mixed angulation $\AA$ with
\begin{itemize}
    \item $p$ black points $\PP^+$ and $q$ white points $\PP^-$,
    \item $b=\alpha$ arcs.
\end{itemize}

Let $\PP^+ := \{ x_1, \cdots, x_p\}$ and $\PP^-:=\{ y_1, \cdots, y_q\}$. When a bi-colored graph $\mathcal{T}$ is given, the expected angle at every vertex is $1$. Recall that the expected ratio is always $q/p$. Hence we need to find a weight function $\mathcal{W}:\AA \to \RR_{>0}$ satisfying the following balance equations: 
\begin{align*}
\sum_{e \in E(x_i)} \mathcal{W}(e) &= 1, \qquad\quad\  \forall x_i \in \mathbb{P}^{+} \ , \\
\sum_{e \in E(y_j)} \mathcal{W}(e) &= \frac1R =\frac{p}{q}, \quad \forall y_j \in \mathbb{P}^{-} \ .
\end{align*}
For convenience, we multiply $\mathcal{W}$ by $q>0$, then the balance equations for this case read as
\begin{align}\begin{split} \label{eq:BE_pq}
\sum_{e \in E(x_i)} \mathcal{W}(e) &= q, \quad \forall x_i \in \mathbb{P}^{+} \ , \\
\sum_{e \in E(y_j)} \mathcal{W}(e) &= p, \quad \forall y_j \in \mathbb{P}^{-} \ .
\end{split}\end{align}

\bigskip
\subsection{Weighted bi-colored planar tree}\label{ssec:weight_bicolor_tree}
The genus zero case is a little bit subtle. We need some preparation before its proof.
Since $g=0$, the bi-colored graph $\mathcal{G}$ is actually a planar graph. However, since
$$ \abs{\AA} = \alpha = p+q+2g-1 = p+q-1 = \abs{\PP} -1 \ , $$
this planar graph is actually a planar tree with $(p+q)$ nodes. 
This can also be obtained by the Jordan theorem, since its complementary is a single simply-connected region. 

\begin{lemma}\label{lem:unique_BE_tree}
    For a given bi-colored planar tree $\mathcal{T}=(\mathbb{P}^{+} \sqcup \mathbb{P}^{-}, \AA)$ as above, the solution of weight function $\mathcal{W}$ to the balance equation \eqref{eq:BE_pq} is unique, if there is any. Furthermore, any solution takes values within $\{1,2, \cdots, q\}$.
\end{lemma}
\begin{proof}
    Assume \eqref{eq:BE_pq} has a positive solution $\mathcal{W}:\AA \to \mathbb{R}_{>0}$. The uniqueness can also be obtained by the rank of connection matrix, introduced in Section \ref{ssec:cnnt_matrix_hcmu}. But we shall provide an alternative proof for later usage.

    Since $\mathcal{T}$ is a planar connected tree, it must contains at least 1 (actually 2) vertex of degree 1. Let $v \in \mathbb{P}$ be a point with $\deg(v)=1$, and let $E(v)=\{e\}$. By \eqref{eq:BE_pq}, $\mathcal{W}\leq q$. Hence $v$ must be a black point with $\mathcal{W}(e)=q$. 

    Now delete all vertices in $\mathbb{P}$ of degree 1, and the edges in $\AA$ connecting to them. Denote the resulted graph by $\mathcal{T}^{(1)} := \left(\mathbb{P}^{(1)}, \mathbb{A}^{(1)}\right)$. 
    Equivalently, $\mathcal{T}^{(1)}$ is the full subgraph spanned by $\mathbb{P}^{(1)} := \defset{v \in \mathbb{P}}{\deg(v)>1}$. 
    Since $\mathcal{T}$ is a connected planar tree, so does $\mathcal{T}^{(1)}$. And the number of vertices and edges of $\mathcal{T}^{(1)}$ is strictly smaller than $\mathcal{T}$.
    As before, there exists at least one $v\in\PP^{(1)}$ with $\deg_{\mathcal{T}^{(1)}}(v)=1$. Here $\deg_{\mathcal{T}^{(1)}}$ means we are counting the degree inside the subgraph $\mathcal{T}^{(1)}$. For any such vertex, let $e_v$ be the unique edge in $\AA^{(1)}$ connecting to $v$.
    Then the \eqref{eq:BE_pq} at that point reads as
    $$ \mathcal{W}(e_v) + \sum_{e\in E(v)\setminus \AA^{(1)}}\mathcal{W}(e) = p \quad \text {or} \quad q $$
    The only unknown weight is $\mathcal{W}(e_v)$. Since $\mathcal{W}(e)$ is positive for all $e\in \AA \setminus \AA^{(1)}$, and we have assumed $\mathcal{W}$ admits a positive solution, we see $\mathcal{W}(e_v)$ is also an integer.

    Similarly, we continue delete all vertex in $\mathcal{T}^{(1)}$ of degree 1 and edges connecting to them. The result graph $\mathcal{T}^{(2)}$ is a connected planar tree again, with strictly less vertices and edges. Then the above argument can be repeated.  
    By induction, this procedure must stop after finite times. Hence $\mathcal{W}$ is unique and always takes integer value.
\end{proof}

\begin{corollary}\label{cor:com_divisor}
    If $\lambda\in\ZZ_{>0}$ is a common divisor of $p,q$, then $\lambda \lvert \mathcal{W}(e)$ for all $e\in\AA$. 
\end{corollary}
\begin{proof}
    Following the previous proof, for all $e\in \AA \setminus \AA^{(1)}$, $\mathcal{W}(e)=q$. 
    In the balance equation \eqref{eq:BE_pq}, every known term is divided by $\lambda$, then so does the only unknown term $\mathcal{W}(e_v)$. 
    The induction holds for the same reason. 
\end{proof}

This leads to the arithmetic condition for genus zero case.
\begin{corollary}\label{cor:necs_S_alpha}
    Condition (2) for the genues zero case in Theorem \ref{thm:S_alpha} is necessary.
\end{corollary}
\begin{proof}
    If there is a weight function $\mathcal{W}$ when $g=0$ and $q \lvert p$, then $q \lvert \mathcal{W}$ by the pervious lemma. 
    Since $\mathcal{W} \in\{1, \cdots q\}$ and $q>1$, $\mathcal{W}$ must be constantly $q$. 
    Now for any $x_i\in\PP^+$, $q= \sum_{e\in E(x_i)} \mathcal{W}(e) = \deg(x_i)\cdot q$. Hence $\deg(x_i)\equiv 1$. 
    
    On the other hand, $q>1$ means there are at least $2$ different white vertices. Because $\mathcal{T}$ is connected, any path in $\mathcal{T}$ connecting 2 different white vertices must contains a black vertex of degree $\geq2$. This is a contradiction.  
\end{proof}

\bigskip
\subsection{Genus zero case}\label{ssec:clsfy_g_0}
When $q=1$, the choice for $\mathcal{T}$ is unique: a white vetex connected by $p$ edges from $p$ black vertices. In fact this is the one we have constructed in Section \ref{ssec:sufficiency}. So we only consider $q>1$ below.

We first consider the case with co-prime $p,q$, which is a number theory problem. Then the remained case can be obtained from co-prime case. It is recommended to browse Figure \ref{fig:tree_7,3} and Figure \ref{fig:tree_14,6} before the proof. 

\begin{proposition}\label{prop:coprime_case}
    For co-prime $(p,q)$, there exists a weighted bi-colored tree $\mathcal{T}=(\mathbb{P}^{+} \sqcup \mathbb{P}^{-}, \AA; \mathcal{W})$ with $\abs{\mathbb{P}^{+}} =p, \abs{\mathbb{P}^{-}}=q$ such that $\mathcal{W}$ satisfies balance equations \eqref{eq:BE_pq}. 
\end{proposition}

\begin{proof}
Line up $p$ black vertices and $q$ white vertices on two parallel lines on the plane. We use $e=(x, y)$ to represent an edge $e \in \AA$ with vertices $x \in \mathbb{P}^{+}, y \in \mathbb{P}^{-}$. 
Now we shall inductively define the edge $e_i$ and is weight $\mathcal{W}(e_i)$, where $1 \leq i \leq p+q-1$. Also let 
$\mathcal{W}_i := \mathcal{W}(e_1) +\mathcal{W}(e_2)+ \cdots + \mathcal{W}(e_i)$. 

First, let $e_1=(x_1, y_1),\ \mathcal{W}(e_1)=q$. Then $q=\mathcal{W}_1 <2q,\ 0\leq \mathcal{W}_1 <p$. If $e_1,\cdots, e_{k-1}$ and $\mathcal{W}(e_1),\cdots, \mathcal{W}(e_{k-1})$ are defined, where $2\leq k< p+q-1$,
then $\mathcal{W}_1, \cdots, \mathcal{W}_{k-1}$ are also known. 
Assume
\begin{align*}
I \cdot q \leq & \mathcal{W}_{k-1} < (I+1) \cdot q, \\
J \cdot p \leq & \mathcal{W}_{k-1} < (J+1) \cdot p,
\end{align*}
with $I \in \{1,2, \cdots, p-1\},\ J \in\{0,1, \cdots, q-1 \}$. 
The two equality on the left can not hold simultaneously unless $k=p+q$ when all edges and weights are defined. Now we define the $k$-th edge to be $e_k := (x_{I+1}, y_{J+1})$, and
$$ \mathcal{W}(e_k) := \min \left\{ (J+1)p - \mathcal{W}_{k-1},\quad (I+1)q - \mathcal{W}_{k-1} \right\} \ . $$
The two terms are different. It can be checked directly that $\forall 1 \leq k < p+q-1$, $e_k, e_{k+1}$ share exactly one common vertex. Hence the tree constructed is connected. The choice of $\mathcal{W}$ guarantee the balance equation \eqref{eq:BE_pq}. 

See Figure \ref{fig:tree_7,3} as an example with $(p,q)=(7,3)$. 
\end{proof}

\begin{figure}[t]
    \makebox[\textwidth][c]{\includegraphics[width=0.6\textwidth]{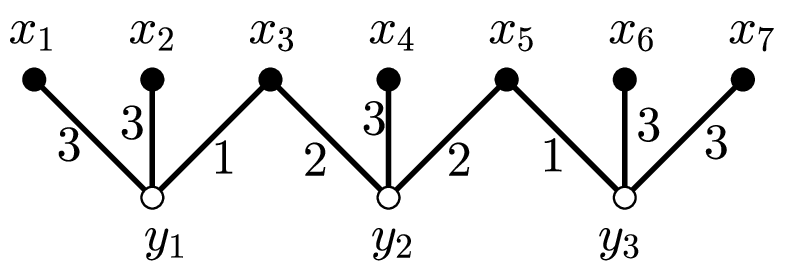}}
    \caption{\footnotesize The bi-colored tree with $(p,q)=(7,3)$ constructed in Proposition \ref{prop:coprime_case}}
    \label{fig:tree_7,3}
\end{figure}

\begin{proposition}\label{prop:non_coprime_case}
    If the greatest common divisor of $p,q$ is $1<\lambda<q$, then there also exists a weighted bi-colored tree $\mathcal{T}=(\mathbb{P}^{+} \sqcup \mathbb{P}^{-}, \AA; \mathcal{W})$ with $\abs{\mathbb{P}^{+}} =p, \abs{\mathbb{P}^{-}}=q$ such that $\mathcal{W}$ satisfies balance equation \eqref{eq:BE_pq}. 
\end{proposition}

\begin{proof}
Let $p=\lambda \overline{p},\ q=\lambda \overline{q}$ with $\overline{p}>\overline{q}>1$ co-prime, and $\overline{m} := \overline{p}+\overline{q}$.

Let $\mathcal{T}_{\overline{m}}=(\overline{\PP}^+ \sqcup \overline{\PP}^-, \overline{\AA}; \overline{\mathcal{W}})$ be the weighted bi-colored tree constructed in the last proposition. The degree of first and last black vertex in $\overline{\PP}^+$ must be $1$, due to the construction. 
Meanwhile, the degree of black vertices must be smaller than 3.
Otherwise at least one of the edge from it will connect to a degree 1 white vertex, which violates the balance equation. 
Since $\overline{q}>1$, there must be a black vertex with degree $\geq 2$, as pointed out in the proof of Corollary \ref{cor:necs_S_alpha}. Hence $\overline{\mathbb{P}}^{+}$ always contains a degree 2 black vertex. 
    
Let $x_I$ be one of them, $I \in\{2, \cdots, \overline{p}-1\}$, and $e_K,\ e_{K+1}$ be the two edges connecting $x_I,\ 1 < K < K+1 < \overline{m}-1$. 
Then $\overline{\mathcal{W}}(e_K) + \overline{\mathcal{W}}(e_{K+1})=\overline{q}$. 
By the connectedness of $\mathcal{T}_{\overline{m}}$, $x_I$ divide it into 2 components, which are trees as well. Let $I^+$ be the one containing $e_K$, and $I^-$ be the other one containing $e_{K+1}$. 
Note that $I^+$ always contains the degree 1 vertex $x_1$, and $I^-$ contains $x_{\overline{p}}$.

Now we construct the desired tree by induction. 
First replace $x_{\overline{p}}$ by two different black vertices $x_{\overline{p}}^{ \pm}$ and split $e_{\overline{m}-1} = (x_{\overline{p}}, y_{\overline{q}})$ into two edges 
$$ e_{\overline{m}-1}^\pm := (x_{\overline{p}}^{ \pm}, y_{\overline{q}}) $$
and define 
$$ \overline{\mathcal{W}}(e_{\overline{m}-1}^+) := \overline{\mathcal{W}}(e_{K}),\quad  \overline{\mathcal{W}}(e_{\overline{m}-1}^-) := \overline{\mathcal{W}}(e_{K+1}) . $$
Next, piece a copy of $I^-$ by gluing $x_I$ to $x_{\overline{p}}^+$, and piece a copy of $I^+$ by gluing $x_I$ to $x_{\overline{p}}^-$. 
We end up with a connected tree $\mathcal{T}_{2\overline{m}}$ with $2\overline{p}$ black and $2\overline{q}$ white vertices. 
The weights are inherited from $\overline{\mathcal{W}}$ naturally, then the balance equation with parameter $(\overline{p}, \overline{q})$ is still satisfied.

Now that $\mathcal{T}_{2\overline{m}}$ must contains a degree 1 black vertex. Hence the splitting and piecing procedure above can be applied repeatedly. By induction, after $(\lambda-1)$ times of operation, we obtain a connected tree $\mathcal{T}_{\lambda\overline{m}}$ with $\lambda \overline{p}=p$ black and $\lambda\overline{q}=q$ white vertices. 
The inherited weight function $\overline{\mathcal{W}}$ satisfies the balance equation with parameter $(\overline{p}, \overline{q})$. 
Then multiply $\overline{\mathcal{W}}$ by $\lambda$ gives a solution to the original balance equation \eqref{eq:BE_pq}. 
See Figure \ref{fig:tree_14,6} for an example from $(\overline{p},  \overline{q})=(7,3)$ to $(p, q)=(14,6)$.
\end{proof}

\begin{figure}[t]
    \makebox[\textwidth][c]{\includegraphics[width=\textwidth]{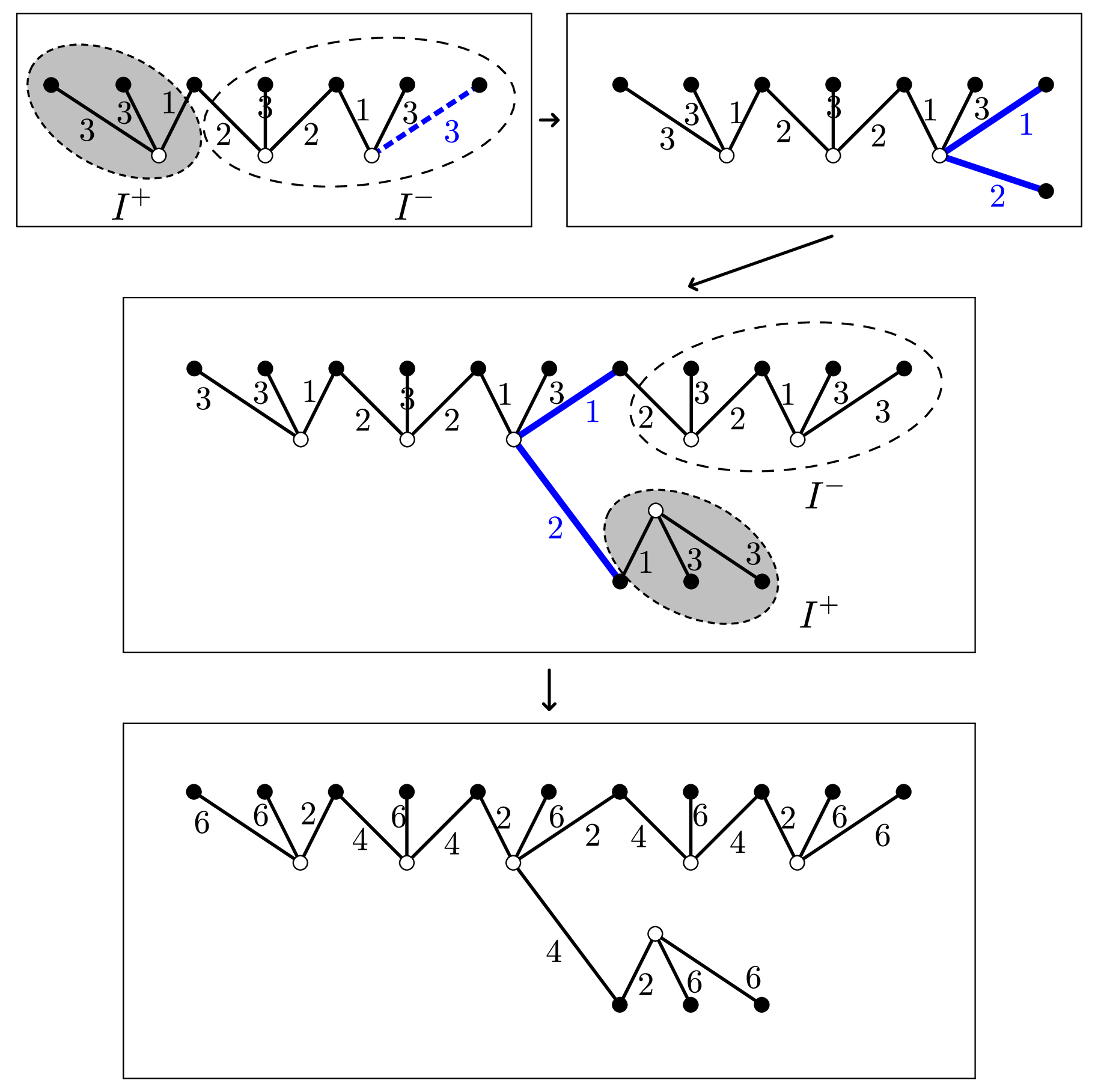}}
    \caption{\footnotesize An example with $(p,q)=(14,6)$ built from $(\overline{p}, \overline{q})=(7,3)$.}
    \label{fig:tree_14,6}
\end{figure}

Now the genus zero case of Theorem \ref{thm:S_alpha} is a combination of Lemma \ref{lem:S_alpha.1}, Corollary \ref{cor:necs_S_alpha}, Proposition \ref{prop:coprime_case} and \ref{prop:non_coprime_case}. 

\medskip
If the weight function $\mathcal{W}$ is multiplied by some positive real number $\lambda>0$, then the cone angle at all extremal points will be multiplied by $\lambda$ at same time. All other parts in the data set is unchanged. Hence we have the following corollary of genus zero case of Theorem \ref{thm:S_alpha}. 
This is not contained in Theorem \ref{thm:angle_cnstr_hcmu} or \ref{thm:angle_cnstr_hcmu_new}, because there are more minimum points. 
\begin{corollary}
    Let $\lambda>0$, integer $p>q>0$ and $\vec{\alpha}$ be an angle vector given by
    \[ \vec{\alpha} = (\ p+q-1\ ,\ \underbrace{\lambda\ ,\ \cdots \ ,\  \lambda}_{(p+q)\textrm{ times}}\ ) \ . \] 
    Also let $\vec{T}= (\{1\}, \{2,\cdots,p+1\}, \{p+2,\cdots,p+q+1\} )$ be a type partition. 
    Then $\Mhcmu{0,p+q+1}{\vec{\alpha};\vec{T}}$ is non-empty if and only if $q=1$ or $q\geq2,\ q \nmid p$.
\end{corollary}

\begin{remark}
    Even when $(p,q)$ is given, the choice of $\mathcal{T}$ is not unique. See the example of $(7,3)$ below. This implies that even if the number of extremal points are fixed, the moduli space of such HCMU surfaces is still not connected. 
    Detailed study on components of such moduli space involves an enumeration of weighted bi-colored trees. Such topics are discussed in \cite{Kyy13, Zak13}. 
\end{remark}

\begin{figure}[ht]
    \makebox[\textwidth][c]{\includegraphics[width=0.8\textwidth]{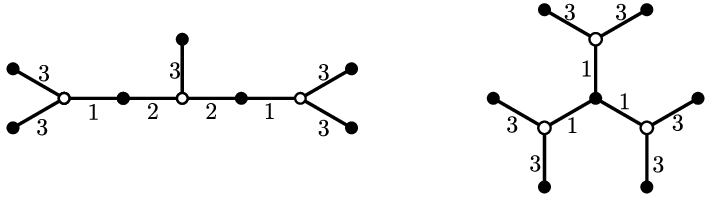}}
    \caption{\footnotesize The two choices of $\mathcal{T}$ when $(p,q)=(7,3)$.}
    \label{fig:tree_7_3_all}
\end{figure}

\bigskip
\subsection{Case of positive genera}\label{ssec:clsfy_g_pstv} 
As before, $q=1$ is actually the case in Section \ref{ssec:sufficiency}. So we always assume $q>1$. 
To construct a surface in $\Mhcmu{g,1}{\alpha}$ with $p$ maximum and $q$ minimum points, we need to construct a $(2\alpha-2)$-mixed angulation with
$\abs{\PP^{+}}=p,\ \abs{\PP^{-}}=q$ and $b=\alpha=p+q+2g-1$ arcs. 
Equivalently, we need to find a weighted bi-colored graph 
$\mathcal{G}=(\mathbb{P}^{+} \sqcup \mathbb{P}^{-}, \AA; \mathcal{W})$ on genus $g>0$ surface with 
\begin{itemize}
    \item $\abs{\PP^{+}}=p,\ \abs{\PP^{-}}=q$, 
    \item $b=\alpha=p+q+2g-1$ arcs,
    \item a unique complementary $(2\alpha)$-gon,
    \item $\mathcal{W}$ satisfying the balance equation \eqref{eq:BE_pq}.
\end{itemize}

\medskip
\begin{proposition}
    If $q \nmid p$, the desired weighted bi-colored graph $\mathcal{G}$ always exists.
\end{proposition}
\begin{proof}
Let $\overline{\mathcal{T}} = ( \overline{\mathbb{P}}^{+}\sqcup \overline{\mathbb{P}}^{-}, \overline{\AA}, \overline{\mathcal{W}})$ be the planar tree constructed in Proposition \ref{prop:coprime_case} or \ref{prop:non_coprime_case}, 
with $\deg(\overline{x}_p)=1,\ \overline{x}_p \in \overline{\mathbb{P}}^{+},\ 
\overline{e}_0=(\overline{x}_p, \overline{y}_q) \in \overline{\AA}$. 
Then $\overline{\mathcal{W}}(\overline{e}_0)=q$. 

On the other hand, let $\AA_g^o$ be the canonical $(4g)$-mixed angulation in Lemma \ref{lem:canonical_polygon}, with $x, y$ be the unique black and white vertex.

Then we draw the subtree 
$\overline{\mathcal{T}} \setminus \{\overline{x}_p, \overline{e}_0\}$ 
on the unique complementary polygon of $\AA_g^o$, such that $\overline{y}_q$ coincident with $y$, and all the other vertices and edges lie in the interior. 
Let $\mathbb{P}^{+} := \left( \overline{\mathbb{P}}^{+} \setminus \{ \overline{x}_p \} \right) \sqcup \{x\},\ \mathbb{P}^{-} \cong \overline{\mathbb{P}}^{-}$ and 
$\AA := \AA_g^o \sqcup \left( \overline{\AA} \setminus \{\overline{e}_0\} \right)$. 

Viewing on the surface, we have $(p-1)+1=p$ black and $q$ white vertices, $(p+q-2)+(2g+1)=\alpha$ arcs. The complementary polygon has
$$ (4g+2) + 2 (p+q-2) = 2(p+q)+(4g-2) = 2 \alpha $$
sides. 
The weights on edges of the subtree is inherited from $\overline{\mathcal{T}}$. 
For all $(2g+1)$ edges of $\AA_g^o$, define its weight to be $q / (2g+1)$. Then the balance equation at $x$ and $y$ are easy to verify.
See Figure \ref{fig:graph_4_3.g=1} for such construction. 
\end{proof}
\begin{figure}[ht]
\makebox[\textwidth][c]{\includegraphics[width=\textwidth]{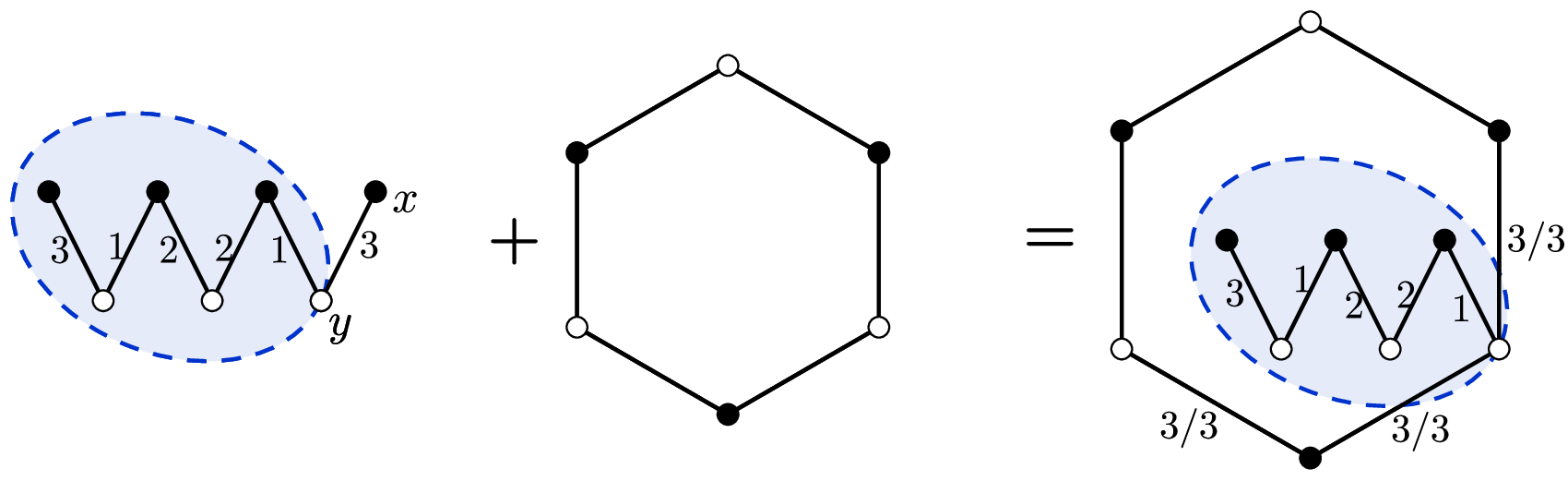}}
\caption{\footnotesize An example for $(p,q)=(4,3)$ with $g=1$.}
\label{fig:graph_4_3.g=1}
\end{figure}

\begin{proposition}
    If $p = k q$ for some $k\in\ZZ_{>1}$, then the desired weighted bi-colored graph also exists.  
\end{proposition}
\begin{proof}
    The point is that we are finding a graph rather than a tree, so there are more freedom on drawing edges and more possibility for the existence of a weight function. The idea is to connect several trees by extra arcs. See Figure \ref{fig:graph_6_2.g=2}. 

    Let $\mathcal{T}$ be the unique planar bi-colored tree with $k$ black vertices and 1 white vertex. As before, assume the black vertices $x_1,\cdots, x_k$ lies on a line in order.
    We still use the regular polygon representation of genus $g$ surface in Lemma \ref{fig:canonical_poly}. But the vertices are not the ones in our desired graph. 

    Fix a regular $(4g+2)$-gon, whose vertices are counterclockwise labeled as $A_1, A_2, \cdots, A_{2g+1}, B_1, B_2, \cdots, B_{2g+1}, A_1$. Here $A_i A_{i+1}$ is parallel to $B_{i+1} B_i$ for $i=1,\cdots, 2g$, and $A_{2g+1} B_1$ is parallel to $A_1 B_{2g+1}$. 
    By rotation, we assume $A_{2g+1} B_1, A_1 B_{2g+1}$ are vertical, with $A_1 B_{2g+1}$ on the left.  
    Let $C_i,D_i$ be the middle point of $A_i A_{i+1}, B_i B_{i+1}$ for $i=1,\cdots, 2g$, and $C_0, D_0$ be the middle point of $A_1 B_{2g+1}, A_{2g+1} B_1$ respectively. $C_i, D_i$ glues to a same point on surface. 
    
    \medskip
    The required weighted bi-colored graph will be constructed in 4 steps. See Figure \ref{fig:graph_6_2.g=2} as an example and outline of this construction. 
    
    First, line up $q$ copies $\mathcal{T}^{(1)}, \cdots, \mathcal{T}^{(q)}$ of the tree $\mathcal{T}$ in the rectangle $\diamondsuit := A_1 A_{2g+1} B_1 B_{2g+1}$ inside the regula polygon. 
    For any $j=1,\cdots, q$, the black vertices of $\mathcal{T}^{(j)}$ are $x_i^{(j)}\ (i=1,\cdots,k)$, and the white vertex is $y^{(j)}$. Let $e_i^{(j)} := (x_i^{(j)}, y^{(j)})$ in $\mathcal{T}^{(j)}$. 
    Then we have $q$ white vertices, $kq=p$ black vertices and $kq=p$ edges. 

    Secondly, add $q$ arcs connecting these $q$ trees. 
    For $j=1,\cdots q-1$, define $e_{k+1}^{(j)}$ to be an edge connecting $y^{(j)}$ and $x_1^{(j+1)}$, lying completely inside $\diamondsuit$ and disjoint from all edges. 
    Besides, define $e_{k+1}^{(q)}$ to be an edge starting from $y^{(q)}$, leaving the polygon through the middle point $D_0$, entering it through another middle point $C_0$, and ending at $x_1^{(1)}$. 

    Thirdly, add $2g-1$ arcs connecting $x_k^{(1)}$ and $y^{(1)}$ to cut the surface into one simply-connected region.  
    For $i=1,\cdots 2g-1$, define $\hat{e}_i$ to be an edge starting $x_k^{(1)}$, leaving the polygon through $C_i$, entering through $D_i$ and ending at $y^{(1)}$. 
    It can be checked directly that the complementary region is simply-connected, by going around the two vertices of the regular polygon.  
    The total number of edge is indeed 
    $ p + q + 2g-1 = \alpha $.  

    Finally, we provide a solution of weight function to balance equations. 
    Let $\mathcal{W}(e_1^{(j)}) = \mathcal{W}(e_{k+1}^{(j)}) := q/2$ for all $j=1,\cdots,q$. 
    And let $\mathcal{W}(e_i^{(j)}) := q$ for all $ 1\leq j \leq q, 2\leq i \leq k$, except $e_k^{(1)}$.
    Define $\mathcal{W}(e_k^{(1)}) := q-1$ (recall that we always assume $q>1$). 
    For the remained edges, let $\mathcal{W}(\hat{e}_i) := 1/(2g-1) >0$, $i=1,\cdots 2g-1$. 
    One can check that the balance equation is satisfied. In particular, the weights on the last $2g-1$ edges recover the missing number 1 on $e_k^{(1)}$.  
\end{proof}

\begin{figure}[h]
    \makebox[\textwidth][c]
    {\includegraphics[width=\textwidth]{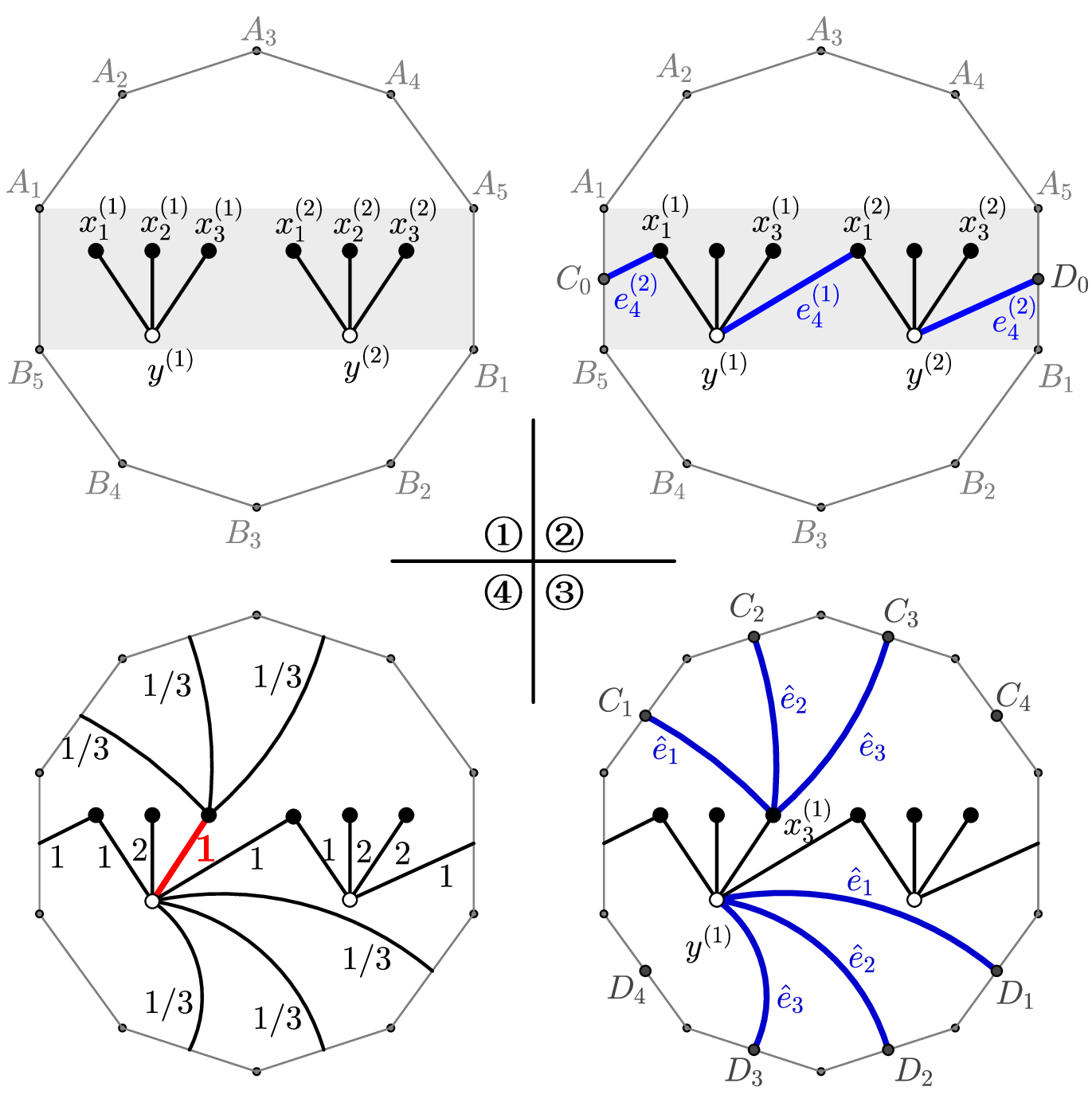}}
    \caption{\footnotesize An example for $(p,q)=(6,2)$ with $g=2$.}
    \label{fig:graph_6_2.g=2}
\end{figure}

\bigskip
\section{Dimension count for moduli space}\label{sec:dim_hcmu}
Now we are ready to finish the dimension count. 
Similar to the angle constraint, Theorem \ref{thm:dim_hcmu} can be obtained by the following theorem on refined moduli space. 

\begin{theorem}
\label{thm:dim_hcmu_T}
    Let $\alpha$ be an angle vector and $Z$ a fixed non-empty subset of $\{1,\cdots,k\}$ with $\abs{Z}=j_0>0$, following the convention and conditions of Theorem \ref{thm:angle_cnstr_hcmu_new}.     
    Then for any possible type partition $\vec{T}=(Z,P^+,P^-)$ with non-empty refined moduli space, the real dimension $\textbf{dim}$ of $\MhcmunT$ is given by the follows.
    \begin{enumerate}[(A). ] 
        \item If there is no cusp, then $\textbf{dim}=2g+2j_0$. 
        \item If there are $q>0$ cusp, then $\textbf{dim}=2g+2j_0+q-1$.
    \end{enumerate}
\end{theorem}

As mentioned in the introduction, the dimension indicates the independent geometric parameters that determine the isometric classes of HCMU surfaces. 
In the data set representation $\left(\AA, \PP^{+} \sqcup \PP^-; K_0, R; \mathcal{W}, \mathcal{L} \right)$, the first two data are discrete topological ones, hence do not contribute to the dimension. We need to count the degree of freedom of the last 4 continuous valued data, when the cone angle and type partition is prescribed. 

All over this section, we shall fix a non-empty refined moduli space $\MhcmunT$ with type partition $\vec{T}=(Z,P^+,P^-)$ and non-empty $Z\subset\{1,\cdots,k\}$.

\bigskip
\subsection{Generic HCMU surfaces are represented by database}\label{ssec:twist} 
In this subsection we first show that generic HCMU surfaces indeed form a maximal dimension subset of the moduli space. 
Hence the data set representation is enough for us to count the dimension. 
We achieve this by the twist deformation introduced below.

\medskip
Transverse to the strip decomposition, an HCMU surface can also be canonically decomposed into topological annuli. 
Let $S_{\alpha}(K_0, K_1)$ be an HCMU bigon, parameterized as Definition \ref{defn:hcmu_bigon}. 

\begin{definition}
    An \DEF{HCMU annulus} \index{HCMU!- annulus} is a subset of $S_{\alpha}(K_0, K_1)$ with $(v,\phi)$-coordinate
    \[ [a,b] \times [0, 2\pi\alpha] \subsetneq [0,l] \times [0, 2\pi\alpha] \ . \] 
    Here $[a,b]\subsetneq [0,l]$ is always assumed to be proper and other than a single point. It is allowed to be a disk when $a=0, b<l$ or $a>0, b=l$. 
\end{definition}

The level sets of the curvature function $K$ form a foliation transverse to all the meridians. As Proposition \ref{prop:vector_fields}, if an integral curve of the Killing vector field $\vec{V}$ does not meet any critical points of $K$, then it must be a connected component of a level set of $K$. By compactness, it must be a circle. 
Let \SYM{$\mathrm{Crti}(\vec{V})$} be the set of all integral curves of $\vec{V}$ passing critical points of $K$, including the extremal points themselves. 
Then $\mathrm{Crti}(\vec{V})$ cuts the HCMU surface into finite number of annuli. This partition is called the \DEF{annulus decomposition} of an HCMU surface.
It can also be obtained by starting with the strip decomposition, further cutting them along $\mathrm{Crti}(\vec{V})$, and then gluing the pieces along meridian segments. 

As the case for strip decompostion in Section \ref{ssec:hcmu_strip_decomp}, a saddle point of angle $2\pi n\ (n\in \ZZ_{>1})$ is glued from $2n$ marked points on the boundary of the annulus.

\medskip
Now assume that $(M,g)$ is an HCMU surface other than a football, and $A$ is an HCMU annulus other than a disk in its annulus decomposition. Geometrically, we regard $A$ as a proper subset parametrized as $[a,b] \times [0, 2\pi\alpha]$ of some football $S_\alpha(K_0,K_1)$. 
Now for an arbitrarily chosen $c\in(a,b)$, let $\gamma=\gamma_A : [0,\alpha] \to A$ be the closed curve defined by $\gamma(t)=\{c\} \times \{ 2\pi t\} \in (a,b)\times[0, 2\pi\alpha]$. Here $\gamma(0)=\gamma(\alpha)$ on the surface. 
Cut $M$ along $\gamma$ and regard the result as a closed surface $M_\gamma$ (possibly disconnected) with two boundaries $\gamma^\pm$. 
Parameterize them as $\gamma^\pm(t)=\gamma(t)$ in the original coordinate of $A$. 
Here $\gamma^+$ is chosen such that the surface lies on the right when $t$ grows from $0$ to $\alpha$. 
Then $ M = M_\gamma \big{/} \gamma^+(t) \sim \gamma^-(t)$. 

\begin{definition}\label{def:twist}
    Let $\psi\in[0,\alpha)$. A \DEF{$\psi$-(right) twist deformation} of $(M,g)$ along annulus $A$ is re-gluing the bordered surface $M_\gamma$ along $\gamma^\pm$ by the identification
    \[ \gamma^+([t]_\alpha) \sim \gamma^-([t - \psi]_\alpha) \]
    for all $t\in [0,\alpha]$. Here for any $t\in\RR$, $[t]_\alpha$ is the real number in $[0,\alpha)$ such that $[t]_\alpha \equiv t (\mathrm{mod}\ \alpha)$. 
    See Figure \ref{fig:hcmu_twist} for illustration. 
\end{definition}


\begin{figure}
      \makebox[\textwidth][c]{\includegraphics[width=1.1\textwidth]{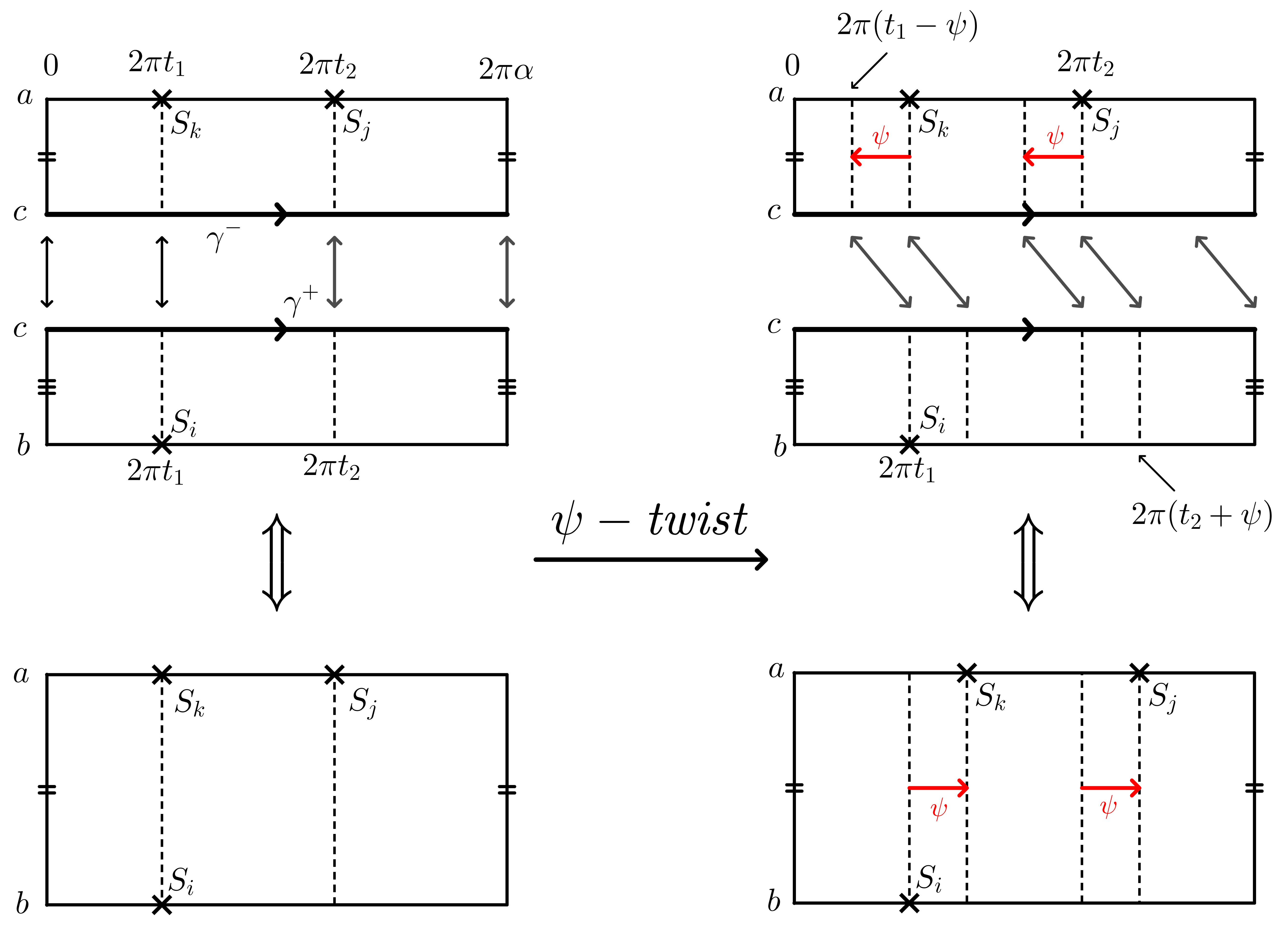}}
      \caption{\footnotesize The effect of twist deformation. The lower row is the HCMU annulus we are twisting along, in its $(v,\phi)$-coordinate as Figure \ref{fig:hcmu_football}. The upper row shows the gluing pattern along $\gamma^\pm$. $S_i, S_j, S_k$ are saddle points on the boundary of this annulus. 
      Also note that the meridian segment connecting $S_i, S_k$ is broken after the twisting. }
      \label{fig:hcmu_twist}
\end{figure}

The gluing operation is available due to the rotational symmetry of HCMU football and annulus. This is equivalent to the invariance of the metric with respect to the translation of $\phi$-parameter. 
The result is independent of the choice of $\gamma\subset A$: the geometry of the open annulus is unchanged due to the rotational symmetry, while the relative positions between the marked points on two different boundaries of the closed annulus are essentially changed.  

This definition resembles the twist in Fenchel-Nielsen coordinates for hyperbolic surfaces. 
Intuitively, the twist deformation alters the alignment of the meridians when they traverse the annulus $A$. Observed from one side of $A \setminus \gamma$, everything on the opposite side appears to shift to the right. 
The underlying Riemann surface is usually changed, while the topology of $\vec{V}$ remains unchanged throughout this process. 
In fact, the twist parameter $\psi$ can take any real value, but the result remains identical if two parameters are congruent modulo $\alpha$, since we are considering isometry classes. 

With absolutely identical argument in \cite[\S4.6]{LuXu24}, we have the follows. 

\begin{proposition}\label{prop:generic_hcmu} ~
\begin{enumerate}[(1). ]
    \item Each non-generic HCMU surface can be continuously deformed to a generic one by a sequence of twist deformation, with fixed cone angles and type partition.
    \item In every non-empty $\MhcmunT$, the non-generic HCMU surfaces form a lower dimensional subset. 
    ~\hfill $\square$
\end{enumerate}
\end{proposition}

The basic idea is to use twist to break any meridian segment connecting two saddle points, and an induction on the number of such meridian segments. Since the twist deformation alters the alignment of the meridians, and there are only finitely saddle points, it is possible to properly choose a twist parameter that breaks that meridian and produces no new segments connecting two saddle points. 
Figure \ref{fig:hcmu_twist} also hints this process, where $S_i, S_k$ are two saddle points connected by a meridian segment in the original surface. The twist deformation break this segment. 
Such parameter also admits a 1-dimensional degree of freedom. Hence we obtain surfaces with more free parameters and less meridian segments connecting points.

\bigskip
\subsection{Choices for ratio are finite}\label{ssec:dscrt_ratio}
Now we investigate the degree of freedom coming from the curvature on this part.

Recall that $m=m(Z)$ in \eqref{eq:extra_pole_hcmu} is the total number of smooth extremal points, and $A^+, A^-$ is the sum of cone angles at maximum and minimum points. 
Also review the notations in Table \ref{tab:extrm_set}. 

\begin{proposition}[Values of ratio]\label{prop:ratio_value}
    When all $\alpha_i>0$, the possible values of the ratio of surfaces in $\MhcmunT$ are discrete and finite. 
    
    More precisely, let $m=m(Z), A^+, A^-$ defined as before, then all the possible choices of $R$ are given by
    \begin{equation}\label{eq:Ratio_value}
        \frac{A^- + m - m^+}{A^+ + m^+} \ , 
    \end{equation}
    where $m^+\in\ZZ_{\geq0}$ satisfying
    \begin{equation}\label{eq:integer_value}
        A^- + m > m^+ > \frac12(A^- + m - A^+) \ .
    \end{equation}
\end{proposition}

\begin{remark}\label{rmk:ratio_value} 
    These expressions are always meaningful for non-empty $\MhcmunT$. 
    \begin{enumerate}[(i). ]
        \item Every HCMU surface contains at least one minimum point. Hence $A^- + m = (A^- + m^-) + m^+$ is always positive in \eqref{eq:integer_value}.
        
        \item $R = \frac{A^- + m^-}{A^+ + m^+} < 1$. Then 
        $A^+ + m > A^- + 2 m^- \geq A^-$ and $\frac12(A^- + m - A^+) < m$. 
        When $A^- > 0$, $m$ is always a possible choice for $m^+$ in \eqref{eq:integer_value}. 
        When $A^- = 0$, $P^- = \varnothing$ and $m^- \geq1$, hence $m\geq1$. Then $A^+ + m > A^- + 2 m^- \geq A^- +2$, and $\frac12(A^- + m - A^+) < \frac12(A^+ + m -2 +m -A^+) = m-1$. 
        Hence $m-1$ is always a possible choice for $m^+$ in \eqref{eq:integer_value}.

        \item If $A^+ = 0$, then $m^+\geq 1$ and $m\geq 1$. Then $m^+ > \frac12(A^- + m -A^+)\geq \frac12$. Hence $A^+ + m^+$ is always positive in \eqref{eq:Ratio_value}.
    \end{enumerate}
\end{remark}

\begin{proof}
    Regard $R$ as a parameter to be determined. Let $m^+$ and $m^-$ be the number of smooth maximum and minimum points, which are unknowns to be solved. 
    By \eqref{eq:extra_pole_hcmu} and \eqref{eq:ratio_given_type}, 
    \begin{equation*}
        \renewcommand\arraystretch{1.3}
        \left\{ \begin{array}{ll}
             A^- + m^- &= R (A^+ + m^+) \\
             m^+ + m^- &= m \\
        \end{array} \right. .
    \end{equation*}
    We have
    \begin{equation}\label{eq:extra_pole_hcmu^pm}
        \renewcommand\arraystretch{1.4}
        \left\{ \begin{array}{ll}
            m^+ & = \frac{1}{1+R}\left( A^- -A^+R + m \right) \\
            m^- & = \frac{1}{1+R}\left( A^+R - A^- + mR \right) \\
        \end{array}
        \right. .
    \end{equation}
    Then $A^- + m - m^+ = R (A^+ + m^+)$, where $m^+$ is regarded as a variable. This gives \eqref{eq:Ratio_value}. 
    Since $0<R<1$, $ 0 < A^- + m - m^+ = R (A^+ + m^+) < A^+ + m^+$. This gives \eqref{eq:integer_value}. 
\end{proof}

The ratio $R$ must be zero when there are cusps, so we can include this case in the above proposition as well. 
Meanwhile, $K_0$ is always free to choose. 
Hence, we say the choice of $(K_0, R)$ is 1-dimensional in the following sense. 
\begin{corollary}\label{cor:choice_K0K1}
    For any non-empty $\MhcmunT$, possibly with cusp singularities, the set (as a subset of $\RR_{\geq0}^2$) of all possible pairs of the maximum curvature and ratio $(K_0, R)$ is a finite union of rays. 
    ~\hfill $\square$
\end{corollary}

\bigskip
\subsection{Connection matrix and balance equations}\label{ssec:cnnt_matrix_hcmu} ~
Now we consider the degree of freedom contributed by weight function $\mathcal{W}$. Let $(M,\rho)\in \MhcmunT$ be a generic surface with $j_0>0$ saddle points and $\left(\AA, \PP^{+} \sqcup \PP^-; K_0, R; \mathcal{W}_0, \mathcal{L} \right)$ be its data set representation. 
Assume $\AA = \left\{e_1, \cdots, e_b \right\}$ is the set of arcs, 
and $\mathbb{P}^{+} = \left\{x_1, \cdots, x_p\right\}, \mathbb{P}^{-} = \left\{x_{p+1}, \cdots, x_{p+q}\right\}$ are the sets of maximum and minimum points, with the cone angle at $x_i$ being $\beta_i\in \RR_{>0}, 1 \leq i \leq p+q$. 
Then the weight function $\mathcal{W}_0$ of $(M, \rho)$ satisfies the balance equations 
\begin{align}\begin{split}\label{eq:BE_beta}
\sum_{e \in E(x_i)} \mathcal{W}_0(e) &= \beta_i, \quad \forall\ 1 \leq i \leq p \ , \\
R\cdot \sum_{e \in E(x_j)} \mathcal{W}_0(e) &= \beta_j, \quad \forall\ p+1 \leq j \leq p+q \ .
\end{split}\end{align}
This subsection will be actually the proof the following: 

\begin{proposition}\label{prop:solv_weight}
    Assume $R\neq0$. When all data in the above is fixed except for $\mathcal{W}$, the solution space of positive weight functions constrained by \eqref{eq:BE_beta} has dimension $(2g+j_0-1)$. 
\end{proposition}

The idea is the same as \cite[\S 5.2]{LuXu24}. We introduce the connection matrix to detect the linear independence of the equations. 
Recall that in Lemma \ref{lem:extrm_set}, $m=m(Z)$ is the total number of smooth extremal points, and $a=a(Z)$ is the total number of extremal points. Note that $p+q = a = n-j_0+m$ in Table \ref{tab:extrm_set}. 
And by \eqref{eq:b(Z)}, $ b=b(Z) = n + m + (2g-2) = \sum_{i \in Z} \alpha_i $ is the total number of edges in $\AA$, or bigons in the strip decomposition. 

\newcommand{\mtxM}{\mathbf{M}}
\newcommand{\mtxR}{\mathbf{\Lambda}}
\begin{definition}\label{defn:cnnct_mtr_hcmu}
    The \DEF{connection matrix} \SYM{$\mtxM$} of a generic HCMU surface $(M,\rho)$ is an $a \times b$ 0-1 matrix, 
    such that $\mtxM_{ij}=1$ if and only if the edge $e_j$ connects to the vertex $v_i$. Equivalently, $\mtxM_{i j}=1$ if and only if the bigon $B_j$ is glued to the extremal point $v_i$ on the surface. 
    The \DEF{ratio matrix} \SYM{$\mtxR$} is the $a \times a$ diagonal matrix
    \begin{equation*}
        \mtxR = \mtxR(R) := \mathrm{diag}\left( \underbrace{1,\ \cdots\ , 1}_{p}\ ,\ \underbrace{R,\ \cdots\ , R}_{q} \right) \ .
    \end{equation*}
\end{definition}

Since each edge connects to two different vertices, each column of $\mtxM$ contains exactly 2 non-zero entries. And the $i$-th row contains $\deg(v_i)$ non-zero entries. 

Let $\vec{\beta}=\left(\beta_1, \cdots, \beta_p, \cdots, \beta_{p+q}\right)^{T}$ be the column vector recording the cone angles at extremal points (we used to call it expected angle vector in Section \ref{ssec:sufficiency}). 
A weight function $\mathcal{W}: \mathbb{A} \rightarrow \mathbb{R}_0$ is also temporarily regarded as a column vector $\vec{\mathcal{W}}=\big(\mathcal{W}(e_1), \cdots, \mathcal{W}(e_b)\big)^{T} \in R_{>0}^b$. All entries of $\vec{\beta}$ and $\vec{\mathcal{W}}$ should be positive. 
Then the balance equations \eqref{eq:BE_beta} can be encoded as a matrix equation 
\begin{equation}\label{eq:width_constraint_hcmu}
    \mtxR \cdot \mtxM \cdot \vec{\mathcal{W}} = \vec{C} \ .
\end{equation}
We shall call it the \DEF{balance matrix equation}. Note that this linear equation still holds when $R=0$. In such case, $\beta_{p+1}=\cdots=\beta_{p+q}=0$ and the only non-trivial equations are the first $p$ rows for the maximum points. 

\newcommand{\rank}{\mathrm{rank}}
\newcommand{\rankM}{\mathrm{rank}(\mtxM)} 
\begin{proposition}\label{prop:rank_hcmu}
    For any generic HCMU surface $(M,\rho)\in\MhcmunT$ with $\abs{Z}\geq1$, $\rankM = a-1$. 
\end{proposition}

\begin{remark}
    By \eqref{eq:b(Z)} and definition \eqref{eq:a(Z)}, when $j_0\geq1$, $b-a=(n+m+2g-2)-(n-j_0+m)=2g-2+j_0\geq-1$. Hence $b\geq a-1$. 
\end{remark}

\begin{proof}
    The argument is the same as \cite[\S 5.2]{LuXu24}. We only sketch the proof here. First, $\mtxM$ is path-connected in the following sense: for any two non-zero elements $\mtxM_{ij}, \mtxM_{kl}$, there exists a finite sequence $(i_0,j_0), (i_1,j_1), \cdots, (i_N, j_N)$ such that
    \begin{itemize}
        \item $(i_0,j_0)=(i,j), (i_N, j_N)=(k,l)$;
        \item $(i_{t-1},j_{t-1}), (i_{t},j_{t})$ have exactly one same component for all $t=1,\cdots,N$;
        \item all entries $\mtxM_{i_{t},\ j_{t}}$ are non-zero.
    \end{itemize}
    This is obtained by the connectedness of surface, similar to the second part of Proposition \ref{defn:weight_bicolor}. 
    Since each column of $\mtxM$ contains exactly 2 non-zero entries, we can deduce that $\rankM \geq a-1$. 
    
    However, $\mtxM$ must not be full rank. Multiply last $q$ rows of $\mtxM$ by $-1$, which correspond to balance equations at white vertices. 
    Since every edge connects a unique black and white vertex, corresponding to the unique +1 and -1 in each column, the sum of the rows in this modified matrix is zero. Then we have $\rankM = a-1$. 
\end{proof}

\begin{proof}[Proof of Proposition \ref{prop:solv_weight}]
    When $R\neq0$, $\mtxR$ is invertible. Proposition \ref{prop:rank_hcmu} shows the solution space of matrix equation \eqref{eq:width_constraint_hcmu} has dimension 
    $$ b - \rankM = b - a+1 = 2g + j_0 -1 \ .$$
    However, this linear subspace already contains a positive solution $\vec{\mathcal{W}}_0$, induced by the predeclared surface $(M,\rho)$. Hence, the solution space of all positive weight functions must have the same dimension. 
\end{proof}

\bigskip
\subsection{Relations between different type partitions}\label{ssec:split_hcmu} 
By definition of type partition, if $\alpha_i\in \ZZ_{>1}$, then the $i$-th conical singularity may also be realized as an extremal point rather than a saddle point. 
Split deformation defined below will replace an extremal cone point of angle $2\pi\alpha_i\ (\alpha_i\in\ZZ_{>1})$ by a saddle point of the same cone angle and $\alpha_i$ new smooth extremal points. This gives us a relation between moduli spaces with different type partition. 

\medskip
Let $p_i$ be a maximum point of cone angle $2\pi\alpha_i\  (\alpha_i \in \ZZ_{>1})$ on $(M,\rho)$. 
Then in the annulus decomposition of $(M,\rho)$, $p_i$ serves as a punctured point of a unique disk $A_i$. $A_i$ must be a part of some football $S_{\alpha_i}(K_0, K_1)$. We assume $A_i$ is parameterized as $(0, d_i) \times [0, 2\pi\alpha_i]$. 

Splitting is a geometric deformation obtained by a cut-and-paste operation.
First pick $\alpha_i \in \ZZ_{>1}$ isometric meridian segments $p_i S_1, \cdots, p_i S_{\alpha_i}$ from $p_i$, equally spaced by $2\pi$ angle. The length $\delta_i$ of these closed segments lies in $(0, d_i)$ so that they contain no singularity other than $p_i$. 
Cutting the surface along $p_i S_1, \cdots, p_i S_{\alpha_i}$, one obtains a bordered surface with one boundary component. The cone point $p_i$ splits into $\alpha_i$ boundary vertices $q_1, \cdots q_{\alpha_i}$. 
Suppose the vertices are labeled as $S_1, q_1, S_2, q_2, \cdots S_{\alpha_i}, q_{\alpha_i}, S_1$ in counterclockwise order. 
Then glue $q_j S_j$ to $q_j S_{j+1}$ for all $j=1,\cdots, \alpha_i$ ($S_{\alpha_i+1} := S_1$). 
Since the cut locus are meridian segments with the same range of $v$-parameter, we obtain another HCMU surface. 
All $S_j$'s are glued to a common point $S$, which is a new saddle point of angle $2\pi\alpha_i$. And each $q_j$ becomes a smooth maximum point. See Figure \ref{fig:split_hcmu}.

For a minimum point of cone angle $2\pi\alpha_i\ (\alpha_i \in \ZZ_{>1})$, a similar procedure can be applied, with $(0, d_i)$ replaced by $(d_i, l)$. $l$ is the length of character line element in Definition \ref{defn:line_element}.

\begin{definition}\label{defn:split_hcmu}
    For an extremal cone point of angle $2\pi\alpha_i\ (\alpha_i \in \ZZ_{>1})$, the process described above is called a \DEF{split deformation}. 
    It turns an extremal cone point into a saddle point and adds $\alpha_i$ smooth extremal points. 
\end{definition}

\begin{figure}
      \makebox[\textwidth][c]{\includegraphics[width=1.1\textwidth]{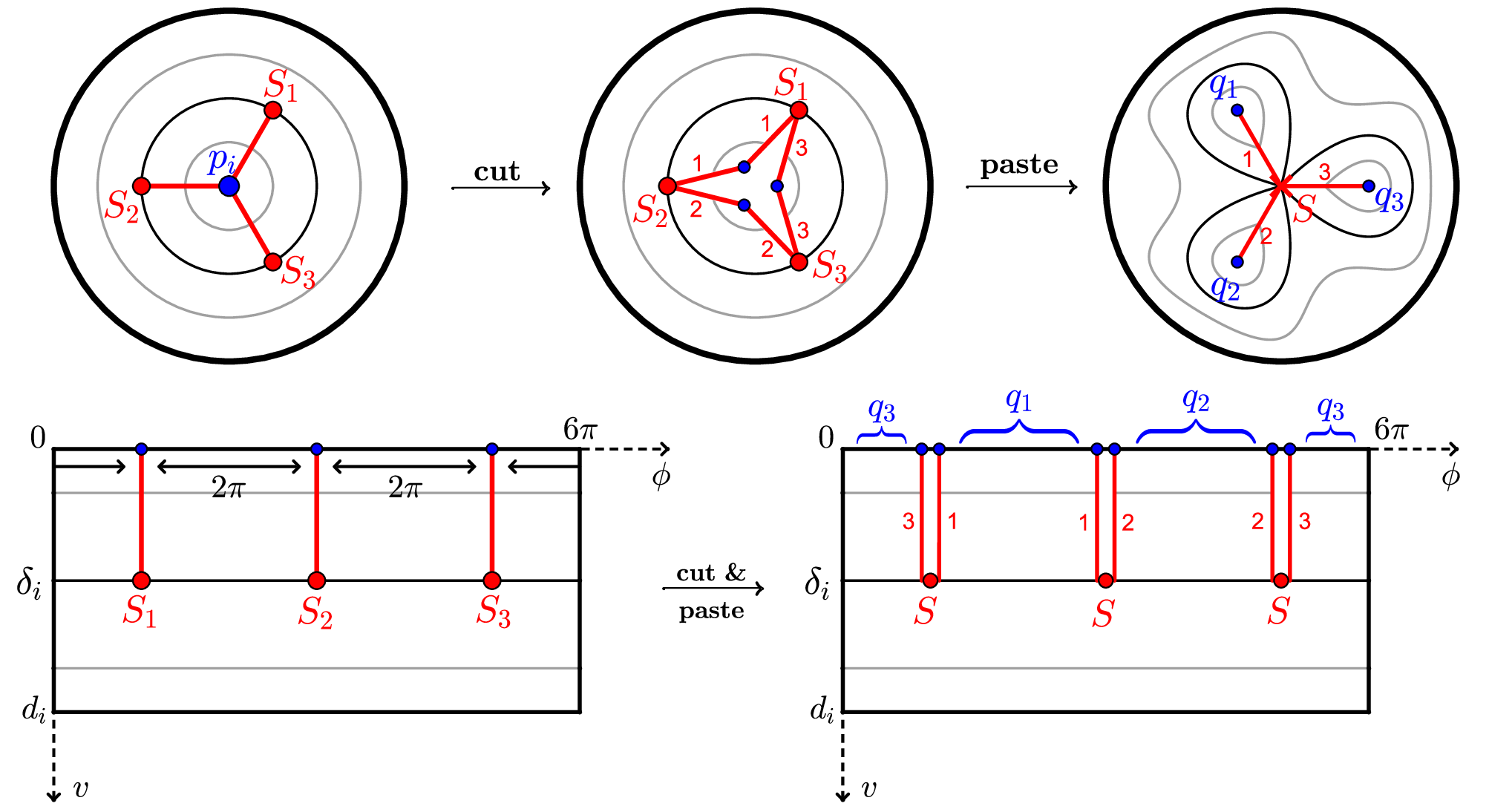}}
      \caption{\footnotesize Splitting an extremal cone point of angle $6\pi$. The upper row shows the local deformation on the surface. Closed black lines are the integral curve of $\vec{V}$. The lower row shows the deformation with respect to the $(v,\phi)$-coordinate. }
      \label{fig:split_hcmu}
\end{figure}

There is a 2-dimensional choice when splitting such an extremal cone point: the $\phi$-parameters of the $\alpha_i$ equally spaced segments, and the length $\delta_i$ of the segments. 
Hence we have the following relation between refined moduli spaces of different type partitions.

\begin{proposition}\label{prop:low_dim_bd_hcmu}
    Let $Z_0:=\{1,\cdots,k\}$ and $Z \subsetneq Z_0$. 
    If $Z\subsetneq Z_0$ in the type partition $\vec{T}$ and $\MhcmunT$ is non-empty, define $\vec{T}_0 := ( Z_0, {P}^+ \setminus Z_0, {P}^- \setminus Z_0)$ by moving all the first $k$ indices from $P^\pm$ into $Z$. 
    Then $\Mhcmu{g,n}{\vec{\alpha}; \vec{T}_0}$ is also non-empty, and $\MhcmunT$ can be regarded as some lower dimensional boundary of $\Mhcmu{g,n}{\vec{\alpha}; \vec{T}_0}$.
\end{proposition}

\begin{proof}
    If there exists $i\in Z_0$ with $i\in P^+ \sqcup P^-$, that is, if $\alpha_i \in \ZZ_{>1}$ but the $i$-th cone point is realized as an extremal point, then define 
    $Z' := Z \sqcup \{i\},\ {P'}^\pm := P^\pm \setminus \{i\}$ and $\vec{T}:=(Z', {P'}^+, {P'}^-)$. 
    Every surface in $\MhcmunT$ can be turned into a surface in $\Mhcmu{g,n}{\vec{\alpha}; \vec{T}'}$ by splitting deformation.     
    As pointed out above, the choice of splitting has a 2-dimensional parameter space. Hence $\Mhcmu{g,n}{\vec{\alpha}; \vec{T}'}$ is also non-empty and has higher dimension. 

    Viewing this deformation in reverse order, the original surface can be regarded as a limiting surface of $\Mhcmu{g,n}{\vec{\alpha}; \vec{T}'}$, with some marked points on the boundary of strip decomposition tending to the vertex of the bigons. 

    The proposition is then obtained by an inductive step. 
\end{proof}

\begin{remark}
    We are not going to discuss the limit of surfaces or the boundary of moduli space more seriously here. This may be a topic for future works.
\end{remark}

\bigskip
\subsection{Dimension count}\label{ssec:dim_count_proof} 
Now we are ready to finish the dimension count. Recall that the dimension equals the maximal number of independent continuous parameters for determining an isometric class of HCMU surfaces. 

\medskip
\begin{proof}[Proof of Theorem \ref{thm:dim_hcmu_T}] ~
    
\textbf{Case (A).} Assume there is no cusp.

By Proposition \ref{prop:generic_hcmu}, we only need to consider generic surfaces. Since the choice of $(\AA, \PP^{+} \sqcup \PP^{-})$ are discrete, they do not contribute to the dimension. In other words, we may fix a choice of $(\AA, \PP^{+} \sqcup \PP^{-})$ which is realized by a generic surface $(M,\rho)\in\MhcmunT$. 

By Corollary \ref{cor:choice_K0K1}, the maximum curvature $K_0$ provides a 1-dimensional continuous parameter, and the choice of $R$ is discrete and finite. 

Recall that $\FF$ is the set of complementary polygons of $\AA$. The distance function $\mathcal{L}$ determines the position of marked boundary  points in each bigon. 
Whenever $(K_0,R)$ is determined, the length $l$ of the common character line element is also determined. Then $\mathcal{L}$, regarded as a vector of dimension $\abs{\FF}=j_0$, can be any point inside $(0,l)^{j_0} \subset \RR_{>0}^{j_0}$. 

Finally, by Proposition \ref{prop:solv_weight}, whenever the bi-colored graph $(\AA, \PP^+\sqcup\PP^-)$ admits a weight function solving the balance matrix equation \eqref{eq:width_constraint_hcmu}, the solution space of all weight functions has dimension $2g+j_0-1$. 
That is, $\mathcal{W}$ has $(2g+j_0-1)$ independent free variables. 

Sunning up all of then, we have
$$ \textbf{dim} = \quad 1 \quad + \quad (j_0) \quad + \quad (2g-1+j_0) \quad = 2g+2j_0 \ . $$

\bigskip
\textbf{Case (B).} Assume there is $q>0$ cusp.

Most part in previous case still works. We only consider generic surfaces with a fixed bi-colored graph $(\AA, \PP^+ \sqcup \PP^-)$, which already admits a weight function $\vec{\mathcal{W}}_{0} \in \RR_{>0}^{b}$. Now $R=0$ is constant while the choice of $K_0$ is still free. And $ \mathcal{L} \in (0,\infty)^{j_0} $ is arbitrary again. 

However, there are less constraints on $\mathcal{W}$. 
$\mtxR(0)$ kills the last $q$ rows and keeps the first $p$ rows of $\mtxM$. They are definitely linearly independent, since each column contains exactly one entry $1$. In other words, $\rank(\mtxR(0)\cdot \mtxM)=p$. 

As before, the solution space of all positive weight function has dimension $b-p$. Note that when there are cusps, $m^-=0$ and $m=m^+,\ \abs{P^-}=q$. Then $p=\abs{P^+} + m^+ = (n-j_0-q)+m$ and 
$$ b-p = (n+m+2g-2) - (n-j_{0}-q+m) = 2g-2 +j_0 +q \ . $$
Summing up all of them, we have 
$$ \textbf{dim} = \quad 1 \quad + \quad (j_{0}) \quad + \quad  (2g-2+j_{0}+q)\quad =2g+2j_{0}+q-1 \ .$$
\end{proof}

\medskip
\begin{proof}[Proof of Theorem \ref{thm:dim_hcmu}]
Case (C) corresponding to footballs. When the angles at extremal points are given, the ratio $R$ is known. The only free parameter is $K_0$.

Case (B) is similar to case (A). We only prove the case without cusp. 

Let $Z_{0}=\{1,\cdots,k\}$. 
By Theorem \ref{thm:dim_hcmu_T}, for any type partition $\vec{T}=(Z_{0}, P^{+}, P^{-})$ with non-empty refined moduli space, the real dimension of $\MhcmunT$ is always $(2g+2k)$. 
Such type partition always exists by Theorem \ref{thm:angle_cnstr_hcmu_new}. 

\medskip
On the other hand, it is still possible for some $Z_1 \subsetneq Z_0$ and type partition $\vec{T}_1=(Z_1, P_1^{+}, P_1^{-})$ inducing non-empty $\Mhcmu{g,n}{\vec{\alpha};\vec{T}_1}$. 
In such cases, let $\vec{T}_0=(Z_0, P_1^{+}\setminus Z_0, P_1^{-}\setminus Z_0)$. By Proposition \ref{prop:low_dim_bd_hcmu}, $\Mhcmu{g,n}{\vec{\alpha};\vec{T}_1}$ is regarded as lower dimensional boundary of $\Mhcmu{g,n}{\vec{\alpha};\vec{T}_0}$, whose dimension is already $(2g+2k)$. 
Hence the original moduli space $\Mhcmun$, possibly as a topological orbifold with boundary, may have several connected components, but all of them have the same dimension 
$(2g+2k)$. 
\end{proof}

\bigskip\noindent
{\bf Acknowledgement. }
The authors would like to express their special thanks to Zhe Han and Yu Qiu for introducing the stories of stability conditions and quadratic differentials, which led them to the most precise language for formulation after much exploration. 
The first author also wish to thank his friend Yisheng Jiang, for conversations on graph theory. Both authors are sincerely grateful to Qing Chen, Zhiqiang Wei and Yingyi Wu for their valuable discussions on HCMU metrics during \emph{The 4th Symposium in Geometry and Differential Equations} held at USTC, July 2024.

\medskip\noindent
{\bf Funding. } 
S. Lu is supported by the China Postdoctoral Science Foundation under Grant Number 2024M753071. 
B. Xu is supported in part by the Project of Stable Support for Youth Team in Basic Research Field, Chinese Academy of Sciences (Grant No. YSBR-001) and National Natural Science Foundation of China (Grant Nos. 12271495, 11971450 and 12071449).

\bigskip
\bibliographystyle{plain}
\bibliography{main}

\end{document}